\documentclass[aos]{imsart}

\RequirePackage{amsthm,amsmath,amsfonts,amssymb}
\RequirePackage[numbers]{natbib}
\usepackage{includes}

\startlocaldefs
\endlocaldefs

\begin{document}

\begin{frontmatter}
\title{Consistency of spectral seriation}
%\title{A sample article title with some additional note\thanksref{T1}}
\runtitle{Consistency of spectral seriation}
%\thankstext{T1}{A sample of additional note to the title.}

\begin{aug}
%%%%%%%%%%%%%%%%%%%%%%%%%%%%%%%%%%%%%%%%%%%%%%%
%% Only one address is permitted per author. %%
%% Only division, organization and e-mail is %%
%% included in the address.                  %%
%% Additional information can be included in %%
\author[A]{\fnms{Amine} \snm{Natik}\ead[label=e1]{amine.natik@umontreal.ca}},
\author[B]{\fnms{Aaron} \snm{Smith}\ead[label=e2]{asmi28@uottawa.ca}}
%%%%%%%%%%%%%%%%%%%%%%%%%%%%%%%%%%%%%%%%%%%%%%
%% Addresses                                %%
%%%%%%%%%%%%%%%%%%%%%%%%%%%%%%%%%%%%%%%%%%%%%%
\address[A]{  Département de Mathématiques et Statistiques\\
  Université de Montréal\\
  Mila - Quebec Artificial Intelligence Institute\\
  6666 St Urbain St, Montreal, Canada.\\ \printead{e1}}

\address[B]{Department of Mathematics and Statistics\\
  University of Ottawa\\
  585 King Edward Ave, Ottawa, Canada. \\
 \printead{e2}}
\end{aug}

\begin{abstract}
Consider a random graph $G$ of size $N$ constructed according to a \textit{graphon} $w \, : \, [0,1]^{2} \mapsto [0,1]$ as follows. First embed $N$ vertices $V = \{v_1, v_2, \ldots, v_N\}$ into the interval $[0,1]$, then for each $i < j$ add an edge between $v_{i}, v_{j}$ with probability $w(v_{i}, v_{j})$. Given only the adjacency matrix of the graph, we might expect to be able to approximately reconstruct the permutation $\sigma$ for which $v_{\sigma(1)} < \ldots < v_{\sigma(N)}$ if $w$ satisfies the following \textit{linear embedding} property introduced in \cite{janssen2019reconstruction}: for each $x$, $w(x,y)$ decreases as $y$ moves away from $x$.  For a large and non-parametric family of graphons, we show that (i) the popular spectral seriation algorithm \cite{atkins1998spectral} provides a consistent estimator $\hat{\sigma}$ of $\sigma$, and (ii) a small amount of post-processing results in an estimate $\tilde{\sigma}$ that converges to $\sigma$ at a nearly-optimal rate, both as $N \rightarrow \infty$. % This is closely related to previous work such as \cite{7924316,RePEc:inm:oropre:v:68:y:2020:i:1:p:53-70} on the natural special case where the graphons belong to the parametric family $w(u,v) = c_{1} + c_{2} \textbf{1}_{|u-v| < c_{3}}$,  which corresponds to the Erd\H{o}s-Renyi graph when $c_{2} = 0$ and to the  Watts-Strogatz ``small world" graph in general.
\end{abstract}



\end{frontmatter}

\section{Introduction}

In this paper, we study a statistical version of the \textit{seriation problem}, first introduced by Robinson in \cite{robinson1951method} as a combinatorial problem. In this problem, we get to observe a graph, and our goal is to order the vertices in such a way that more similar items are placed near to each other. The pioneer of seriation is considered to be Flinders Petrie, an Egyptologist,  who called it \textit{sequence dating} in \cite{petrie1899sequences}.  While the seriation originates from archaeology, it has recent applications to ecology, sociology, biology, and other topics. See \cite{liiv2010seriation} for a historical overview on the seriation problem.

We relate the seriation problem from  \cite{petrie1899sequences} to the notation used in this article. We begin with a collection of $N \in \mathbb{N}$ objects, with the $i$'th object associated with an unobserved point $v_{i} \in [0,1]$. In \cite{petrie1899sequences}, these objects corresponds to a set of graves, with $v_{i}$ the age of the $i$'th grave. We then get to observe a symmetric graph $G = (V, E)$ with vertex set $V = \{v_1, v_2, \ldots, v_N\}$. Informally, we expect that an edge is more likely to appear between vertices $v_i$ and $v_j$ that are closer together. Continuing the example of \cite{petrie1899sequences}, the observed graph is constructed by putting an edge between $v_i$ and $v_j$ if we find the same type of artifacts in both the $i$'th and $j$'th graves. This co-occurrence is more likely when the two graves have similar ages. The \textit{goal} in seriation is then to find an ordering $\sigma \, : \, V \mapsto \{1,2,\ldots,N\}$ for which $v_{\sigma(1)} < v_{\sigma(2)} < \ldots < v_{\sigma(N)}$. In \cite{petrie1899sequences}, this means inferring a chronological ordering of the true ages of the set of graves by observing the different types of archaeological finds on each grave.

Many versions of the seriation problem have been studied. Perhaps the simplest occurs when the observed graph satisfies the following property: for all $v_{i} < v_{j} < v_{k}$, 
\begin{equation} \label{Eq:SeriationProperty}
     (v_i,v_k) \in E  \implies  (v_i,v_j) \in E  \ \text{and} \ (v_j,v_k) \in E.
\end{equation}
Graphs satisfying this property are called \textit{Robinsonian}, named after Robinson who first introduced this property in \cite{robinson1951method}. We call this the \textit{noiseless} seriation problem, as there will never be any ``missing" edges and a perfect ordering exists.

\begin{remark} [Reversals and Recovery]

Since we can only observe the graph $G$, there is no way to distinguish between an ordering $\sigma$ and its reversal $\mathrm{rev}(\sigma)$, defined by
\begin{equation} \label{eq:ReversalOrdering}
    \mathrm{rev}(\sigma)(i) = N+1 -\sigma(i).
\end{equation}
This means that, although we say that we wish to find an ordering $\hat{\sigma}$ that is close to the ordering $\sigma$ that satisfies $v_{\sigma(1)} < v_{\sigma(2)} < \ldots < v_{\sigma(N)}$, the best that we can hope for is an ordering $\hat{\sigma}$ that is close to \textit{either} $\sigma$ \textit{or} $\mathrm{rev}(\sigma)$. We expect that this non-identifiability is rarely a problem in practical applications: even a small amount of additional information lets you distinguish between the extremely different estimates $\hat{\sigma}$ and $\mathrm{rev}(\hat{\sigma})$.

To avoid repeatedly discussing this small non-identifiability issue, we usually hide it by the following small abuse of notation. For any metric $d$ that can be applied to orders, we always intend the symmetrized version
\begin{equation} \label{EqSymmMetric}
d(\sigma_1,\sigma_2) = \min(d(\sigma_1,\sigma_2), d(\mathrm{rev}(\sigma_1), \sigma_2))
\end{equation}

unless we explicitly say otherwise. We only need to pay attention to this issue in step 3 of our post-processing algorithm (Algorithm \ref{AlgPostProc}), where we consider several different estimated orderings and must ``align" them.

\end{remark}

This noiseless version of the seriation problem has been well-studied in the computer science literature, and there are several algorithms that are known to return a correct $\sigma$ very quickly; we believe that \cite{corneil2004simple} is the fastest. Of the provably-correct algorithms, one of the simplest and best-known is the \textit{spectral seriation algorithm} (see Algorithm \ref{Alg:SpectralSeriation} below) introduced in \cite{atkins1998spectral}. This algorithm is based on the spectrum of the \textit{Laplacian} of the observed graph.

In realistic examples, the data is noisy and \eqref{Eq:SeriationProperty} will not hold exactly. This motivates the \textit{statistical seriation problem}, in which the edges of the graph $G$ are seen as random and \eqref{Eq:SeriationProperty} only holds ``on average.'' We study the following natural generalization introduced in \cite{chuangpishit2015linear}: for all $v_i < v_j < v_k$,
\begin{equation} \label{Eq:SeriationPropertyStat}
\Prob{(v_i,v_k) \in E} \leq \min\left(\Prob{(v_i,v_j) \in E}, \, \Prob{(v_j,v_k) \in E}\right).
\end{equation}

In this situation we typically can't recover $\sigma$ exactly, and so the goal becomes \textit{estimating} $\sigma$.  Algorithm \ref{Alg:SpectralSeriation}, the spectral seriation algorithm, is commonly applied to the statistical seriation problem as well as the noiseless seriation problem.

We give the usual spectral algorithm below, but first set notation. For any integer $n\in \mathbb{N}$ define $[n] = \{1,2,\ldots,n\}$. For any set $S$ and any function $f \, : \, S \mapsto \mathbb{R}$, define the associated ordering $\sigma_{f} \, : \, S \mapsto \left[|S|\right]$ by the formula
\begin{equation}\label{Eq:DefFuncToPerm}
    \sigma_{f}(i) = | \{ j \in S \, : \, f(j) \leq f(i) \}|,
\end{equation}
breaking ties in $f$ arbitrarily to make the resulting function bijective. If $S = [N]$ and $f$ does not contain any repeated values, it is clear that $f(\sigma_{f}(1)) \leq f(\sigma_{f}(2)) \leq \cdots \leq f(\sigma_{f}(N))$.  

The spectral seriation algorithm is based on the eigenvalues of a related Laplacian matrix, rather than the adjacency matrix itself. We recall that the (unnormalized) Laplacian matrix of a graph with $N$ vertices and adjacency matrix $A = (a_{i,j})_{1\leq i,j \leq N}$ is given by $L_{A} = (D - A)$, where the \textit{degree matrix} $D = (d_{ij})_{1 \leq i,j\leq N}$ is the diagonal matrix with entries $d_{ii} = \sum_{j=1}^{N} a_{ij}$. The spectral seriation algorithm introduced in  \cite{atkins1998spectral} is as follows:

\begin{algorithm}[H] \label{Alg:SpectralSeriation}
    \SetKwInOut{Input}{Input}
    \SetKwInOut{Output}{Output}

    \Input{ A graph $G$.}
    Compute the adjacency matrix $A$ of the graph $G$.\\
    Compute the Laplacian matrix $L_A$.\\
    Find the eigenvector $\phi$ of $L$ with second-smallest eigenvalue, called the \textit{Fiedler vector.}\\
    Compute the permutation $\sigma_{\phi}$ according to Equation \eqref{Eq:DefFuncToPerm}.\\
    \Output{ The permutation $\sigma_{\phi}$.}
    \caption{The spectral seriation algorithm.}
\end{algorithm}

As shown in \cite{atkins1998spectral}, under the satisfying property \eqref{Eq:SeriationProperty} this algorithm always returns a correct permutation. The same algorithm also gives a popular estimator for the \textit{statistical} seriation problem. The main contribution of this paper is the study of the performance of this algorithm for the statistical seriation problem under fairly general conditions, which we now introduce.

A \textit{graphon} is a measurable function $w : [0,1]^2 \to [0,1]$ which is symmetric (i.e., $w(x,y) = w(y,x)$). See \cite{lovasz2012large} for a general introduction to graphons, and  \cite{chuangpishit2015linear} for previous work on graphons in the context of seriation. In this paper, we say that a random graph $G = (V, E)$ with $N$ vertices is sampled from a graphon $w$ and write $G \sim w$ if: (i) the vertices of the graph $V = \{v_1, v_2, \ldots, v_N\}$ are $v_i = \tfrac{i}{N}$ and (ii) the edges of $G$ are sampled independently with probabilities: \footnote{In the larger graphon literature, it is more common to view $v_{i}$ as being i.i.d. samples taken uniformly from $[0,1]$. This does not make a large difference in the types of estimates we obtain in this paper. We use this embedding to more closely match previous work on special cases of the seriation problem, such as \cite{7924316,RePEc:inm:oropre:v:68:y:2020:i:1:p:53-70,DingSmallWorld}. }
\begin{equation}\label{eq:samplingGraphs}
    \Prob{(v_i, v_j) \in E} = w \left(v_i, v_j\right).
\end{equation}

Our paper includes two main results on this problem. We now give weak but-easy-to-state versions of our main theorems, applying to the following family of random graphs:

\begin{defn} [``Nice" Random Graphs] \label{Assumption:weakfamily}
We say that a graphon $w$ is ``nice" if $w$ can be written as $w(x,y) = R(\abs{x-y})$ for all $x,y\in [0,1]$, where  $R:[0,1]\to(0,1]$ is a $\mathcal{C}^1$ function such that $R'(x)<0$ for all $x\in(0,1)$.
\end{defn}

To give a simple concrete family of examples: (i) the graphons $w(x,y) = 1 - \alpha |x-y|^{\beta}$ belong to this family when $0 < \alpha \leq 1$ and $1 \leq \beta < \infty$, and (ii) the graphons $w(x,y) = \exp{\left(\frac{-(x-y)^2}{2\sigma^2}\right)}$ for any $\sigma > 0$ (known as the RBF kernel) also belong to this family.

Denote by $id_N$ the identity permutation on $[N]$, and for all $\gamma > 1$ and $N \geq 1$ define $\varepsilon_{\gamma}(N) := \exp{\left( -\log^{\gamma}(N)\right)}$. The following is an immediate consequence of Theorems \ref{ThmMainThm} and \ref{ThmSomePostprocMain}, as we show in Section \ref{SecApplNice}:

\begin{thm} \label{ThmSimpleVersion}
Let $w$ be a graphon of the form given in Definition \ref{Assumption:weakfamily}. Consider a family of random graphs  $\{G^{(N)}\}_{N \in \mathbb{N}}$ where each $G^{(N)}$ has exactly $N$ vertices, and is sampled from $w$ in the sense of Equation~\eqref{eq:samplingGraphs} (i.e., $G^{(N)} \sim w$), and let $\hat{\sigma}^{(N)}$ be the output of Algorithm \ref{Alg:SpectralSeriation} with input $G^{(N)}$. For any fixed $\varepsilon \in (0, 1)$ and $\eta \in (0, 1/3)$, there exists a positive constant $C = C(w) < \infty$ and an integer $N_0 = N_0(w, \varepsilon, \eta)$ so that for all $N \geq N_0$
\begin{equation}  \label{IneqThm1Cons}
\Prob{ \norm{\hat{\sigma}^{(N)} - id_N}_1 \leq C N^{2 - \eta} } > 1 - \varepsilon.
\end{equation}
 Furthermore, denote by $\tilde{\sigma}^{(N)}$ the ``post-processed'' estimate obtained from Algorithm \ref{AlgPostProc} with input $G^{(N)}$ and parameters $\alpha = 0.05$, $\beta = 0.31$. Then for any $\gamma >1$ and $\epsilon > 0$, there exist a positive constant $\delta = \delta(w, \gamma, \epsilon)$  and an integer $N_0 = N_0(w, \gamma, \epsilon)$ so that 
\begin{equation} \label{IneqThm2Cons}
\Prob{\norm{\tilde{\sigma}^{(N)} - id_N}_{\infty} \leq  \epsilon \,\sqrt{N \log^{\gamma}(N)}} > 1 - \varepsilon^{\delta}_{\gamma}(N)
\end{equation}
for all $N \geq N_0$.
\end{thm}

We note that Inequality \eqref{IneqThm2Cons}, which is a weak form of Theorem \ref{ThmSomePostprocMain}, is much stronger than \eqref{IneqThm1Cons}, which is a weak form of Theorem \ref{ThmMainThm}, in at least two ways: the norm in which convergence is measured is much stronger, and the rate of convergence is much faster (even after rescaling to account for the different norms). We discuss the degree to which this is rate-optimal, and how to further improve this results under some additional assumptions, in Section \ref{SecOptDisc}.

\subsection{Previous Work and Main Contributions}

Informally, the main results of our paper are that (i) the popular spectral embedding algorithm gives reasonably good results for the seriation problem, and (ii) a simple post-processing step can be added to obtain very good results. We view this as important reassurance for practitioners: spectral embedding algorithms are very popular, but there was little theoretical support for using these algorithms to fit anything but a few small parametric families of models. Our results provide evidence for the folk belief that spectral algorithms, with appropriate coarse post-processing, are quite insensitive to the details of the generating process $w$.

There are many papers in the literature that have either similar-looking results or similar-looking methods, so we give a quick survey and then summarize the most interesting differences.

\subsubsection{Previous Work: Similar Algorithms and Proofs, Different Questions}

Spectral embedding algorithms are very popular in many contexts, especially in the \textit{spectral clustering} literature (see e.g. \cite{donath2003lower, von2007tutorial} and variants such as  \cite{krzakala2013spectral, chatterjee2015matrix,avella2018centrality}). The idea behind all these algorithms is to calculate the first few eigenvectors of a matrix related to your data matrix, use these to embed your data in a low-dimensional space, then use a popular clustering algorithm such as $k$-means to cluster points in this representation. As with the seriation problem, this is known to give exactly the desired answer in simple ``noiseless" situations. 

In ``noisy" situations, the most common analyses follow the approach of  \cite{von2008consistency}. In that paper, the authors show that, under appropriate conditions, the eigenvectors of the observed data matrix will converge to the eigenfunctions of a certain limiting operator. When this convergence happens, the associated spectral embedding will also converge to an embedding related to the limiting operator. One can then bound the difference between the clustering associated with the observed data and the limiting operator, obtaining a consistency result. 

The proof of our first main bound, Inequality \eqref{IneqThm1Cons}, follows essentially this approach. We emphasize two main differences:

\begin{enumerate}
    \item \textbf{Definition of ``correct" answer:} In spectral clustering, it is often straightforward to check that the clustering associated with the limiting operator is ``correct," and sometimes this clustering is even taken as the definition of a ``correct" clustering. In seriation, the correct ordering has an independent definition in terms of the embedded positions $v_{1},\ldots,v_{N}$. This means that we must check that the seriation obtained from the limiting operator actually agrees with this ordering. Checking this turns out to be non-trivial unless the generating graphon $w$ is very special, and relating the limiting operator to the embedded positions takes up most of Section \ref{SecFiedlerProp}.
    \item \textbf{Difficulty of problem:} In spectral clustering, the ``correct" clustering maps vertices to a finite set $\{1,2,\ldots,k\}$, where $k$ does not grow with the size $N$ of the observed graph. In seriation, the ``correct" ordering is a permutation on $\{1,2,\ldots,N\}$. The fact that the seriation problem attempts to estimate a much higher-dimensional object makes the problem substantially messier in a few ways. For example, it is often possible to perfectly recover the correct clustering with a finite sample size even without precise information about the data-generating process (see \textit{e.g.} \cite{daniely2012clustering}), but it is usually not possible to perfectly recover the solution to the seriation problem without having very accurate a priori knowledge of the graphon $w$ (see \textit{e.g.} \cite{janssen2019reconstruction}). 
\end{enumerate} 

We next relate our second main result, Inequality \eqref{IneqThm2Cons}, to the existing literature. 

Although analysis of the spectral embedding gives useful information, there are many possible spectral embeddings and it is difficult to find and analyze the ``best" one (see e.g. the discussion of various Laplacians in \cite{von2008consistency}). An alternative is to use a simple spectral embedding to obtain an initial rough solution, then use a simple post-processing algorithm to refine it. This is known to work well for the stochastic block model, the most popular toy model for the spectral clustering problem \cite{JMLR:v18:16-480}.  

We take a similar approach here to prove Inequality \eqref{IneqThm2Cons}. The main novelty is that we need to develop a post-processing algorithm, as the algorithms used for spectral clustering are not appropriate. Again, the fact that exact recovery is impossible adds some complication to our analysis.

\subsubsection{Previous Work: Similar Questions}

Our work was inspired by  \cite{rocha2018recovering,fogel2014serialrank}, which studied random graphs that were very similar, respectively, to  the one-parameter graphon $w(x,y) = p \, \mathbf{1}_{|x-y|\leq 0.5}$ for some $p \in [0,1]$ and the single graphon $w(x,y) = 1- |x-y|$. Like us, \cite{rocha2018recovering} follows in the footsteps of  \cite{von2008consistency} and analyzes the spectral seriation algorithm as a perturbation of a limiting algorithm. The main differences between our paper and \cite{rocha2018recovering,fogel2014serialrank} are:

\begin{enumerate}
    \item \textbf{Inequality \eqref{IneqThm1Cons}:} We deal with a large nonparametric family of models, while \cite{rocha2018recovering,fogel2014serialrank} are both restricted to small parametric families. Furthermore, we allow for sparser graphs than  \cite{rocha2018recovering} and more noise than \cite{fogel2014serialrank}.
    \item \textbf{Inequality \eqref{IneqThm2Cons}:} Using our post-processing algorithm, we obtain sharper results than previous work (even in the special cases that were focused on).
    \item \textbf{Technical differences:} Both \cite{rocha2018recovering,fogel2014serialrank} rely on detailed knowledge of the spectrum of a limiting operator, which is only possible because they study rather special  families of graph models. Since we deal with the nonparametric case, we do not have access to this information and need substantially different arguments. 
\end{enumerate}

The results in this paper are also closely related to those in \cite{janssen2019reconstruction}, which share an author. Both papers study a large nonparametric family of graphons. The main difference is that the present paper analyzes the very popular and relatively simple spectral seriation algorithm, while \cite{janssen2019reconstruction} develops new and rather complicated algorithms. As a secondary difference, the assumptions in our papers are not comparable. In particular, \cite{janssen2019reconstruction} requires that the \textit{square} of the adjacency matrix satisfy a weaker version of the Robinsonian property, and this does not always hold even for uniformly embedded graphons.

Finally, we note that there is a related literature on detecting and recovering embeddings on the \textit{cycle} rather than the \textit{path}, such as \cite{7924316,RePEc:inm:oropre:v:68:y:2020:i:1:p:53-70,DingSmallWorld}. These papers cover a wider range of algorithms than ours, including some analyses (e.g. of row-correlation) that are quite different from ours. However, they have essentially the same goal. The largest difference is that, like \cite{rocha2018recovering,fogel2014serialrank}, the papers \cite{7924316,RePEc:inm:oropre:v:68:y:2020:i:1:p:53-70,DingSmallWorld} are concerned with parametric families of graph with few parameters, and take advantage of this in developing and analyzing their algorithms.

\subsubsection{Previous Work: Summary and Further Reading}

Our work belongs to a very large literature on spectral embedding methods, and a smaller literature on attempting to order or embed vertices on an interval or cycle. The most important differences are:

\begin{enumerate}
    \item In contrast to most previous work on related questions, we consider a large nonparametric family of graphons. In particular, while our assumptions appear more technical than theirs, this is largely because we are considering a much broader class of models - they are often straightforward to check for the most common explicit parametric families. 
    \item In contrast to previous work on spectral clustering, we must check that the embeddings associated with our limiting operators actually give good solutions to the original problem. 
    \item In contrast to our own previous work  \cite{janssen2019reconstruction}, we consider both a large nonparametric families of graphons \textit{and} a simple, algorithm based on a widely-used approach.
\end{enumerate}

Finally, we mention recent work on several problems that have a similar feel to seriation, but which are not directly comparable to our note. Several papers, such as \cite{flammarion2019optimal} and \cite{recanati2018robust}, have studied what amount to other relaxation of the condition in Equation \eqref{Eq:SeriationProperty}. These relaxations are also reasonable versions of the ``statistical" or ``approximate" seriation problem, though the details are quite different. There have also been many articles on the closely-related problem of estimating a permutation from pairwise comparisons (see \textit{e.g.} \cite{mao2017minimax}) or other estimates related to the entries of a random matrix whose expected value is a Robinsonian matrix. See the longer paper \cite{janssen2019reconstruction} for a longer section relating our statistical seriation question to this work.

\subsection{Paper Guide}

In Section \ref{SecNotation}, we introduce the main pieces of notation used throughout the rest of the paper and formally state Theorems \ref{ThmMainThm} and \ref{ThmSomePostprocMain}, stronger versions of the first and second parts of Theorem \ref{ThmSimpleVersion}. 

In Section \ref{SecFiedlerProp}, we check that a conjectured limit of the graph Laplacian has certain good properties - primarily that  an associated eigenfunction is monotone. In Section \ref{SecConv}, we check that our graph Laplacian really does converge to this conjectured limit in a sufficiently strong sense.  In Section \ref{SecErrorBound}, we put together the bounds in Section \ref{SecFiedlerProp} and \ref{SecConv} to prove Theorem \ref{ThmMainThm}.

In Section \ref{SecPostProc}, we introduce the new post-processing algorithm and prove Theorem \ref{ThmSomePostprocMain}. 

In Section \ref{SecApplNice}, we check that the assumptions of Theorems \ref{ThmMainThm} and \ref{ThmSomePostprocMain} are weaker than the assumptions of  Theorem \ref{ThmSimpleVersion}. Finally, we have a short concluding discussion in Section \ref{SecOptDisc}.

\section{Notation and Main Results} \label{SecNotation}

\subsection{Notation for Random Graphs}

Let $w : [0,1]^2 \to [0,1]$ be a graphon. For each integer $N$,  consider the partition $(\iI_{i}^N)_{i=1,\ldots,N}$ of the interval $[0,1]$ defined by $\iI_{1}^N=[0,1/N]$ and $\iI_{i}^N=((i-1)/N,i/N]$ for $i=2,\ldots,N$. For a given graphon $w$ we define the \textit{model matrix} $P^{(N)}\in[0,1]^{N\times N}$ by $P_{ij}^{(N)}:= w(i/N,j/N)$ for all $i,j\in\left\{1,2,\ldots,N\right\}$.

For each $N \in \mathbb{N}$, we consider the sequence of random graphs $G^{(N)} \sim \rho_N w$, where $\rho_N \in (0,1]$ is a sequence regulating the sparsity of the graph. We denote by  $\widehat{P}^{(N)}$  the adjacency matrix of the graph $G^{(N)}$. We recall that, the graph $G^{(N)}$ has $N$ vertices and the edges are constructed as follows: for each $1 \leq i < j \leq N$, we add an undirected edge between the $i$'th and $j$'th vertices with probability $\rho_N P_{ij}^{(N)}$; these edges are added independently. In other words,
\begin{equation}\label{EqDefSamplingRule}
    \sP^{(N)}_{ij} \stackrel{\text{ind}}{\sim} \mathcal{B}\text{ernoulli}\left(\rho_N w(\frac{i}{N}, \frac{j}{N})\right).
\end{equation}

We view the graphon $\rho_N w$ as the limit of the sampling matrices $\sP^{(N)}$ under suitable rescaling (see \cite{glasscock2015graphon} for an explanation of how this notion of ``graph limit" can be made precise). This motivates the graphon analogue to property \eqref{Eq:SeriationPropertyStat}:
for $x,y,z\in [0,1]$,
\begin{align*}
    y<z<x  &\implies w(x,y)\leq w(x,z), \\
    x<y<z &\implies  w(x,y) \geq w(x,z).
\end{align*}

In analogy to \eqref{Eq:SeriationPropertyStat}, we call graphons with this property \textit{Robinsonian} (these are also what \cite{chuangpishit2015linear} calls \textit{diagonally increasing}).

We define  the sampled graphon $w_N$ by 
\begin{equation*}
    w_N(x,y) := \sum_{i=1}^N\sum_{j=1}^NP_{ij}^{(N)}1_{\iI_i^N}(x)1_{\iI_j^N}(y)
\end{equation*}
for all $(x,y)\in [0,1]^2$.

Similarly, the random graphon $\widehat{w}_N$ is defined by:
\begin{equation*}
    \widehat{w}_N(x,y) := \sum_{i=1}^N\sum_{j=1}^N  \sP_{ij}^{(N)}1_{\iI_i^N}(x)1_{\iI_j^N}(y). 
\end{equation*}

\subsection{Review of Relevant Facts from Spectral Theory}

We next need some basic notation related to operators. We denote by $L^2([0,1])$ the usual Hilbert space on (equivalence classes of) functions $f \, : \, [0,1] \mapsto \mathbb{R}$, with inner
product given by  $\langle f , g \rangle := \int_0^1 f(x)g(x)dx$ and norm denoted $\norm{f}$. Let $e\in L^2([0,1])$ denote the constant function, $e(x) := 1$ for all $x\in[0,1]$. We use blackboard bold symbols, such as $\mathbb{T}$, to denote linear operators acting on $L^2([0,1])$. The induced (operator) norm is defined as
\begin{equation*}
    \normoo{\mathbb{T}} := \sup_{f\in L^2([0,1]) \ \mathrm{s.t.} \ \norm{f}=1} \norm{\mathbb{T}f}.
\end{equation*}

For a self-adjoint bounded linear operator $\mathbb{T}$ we denote $\sigma(\mathbb{T})$ the spectrum of $\mathbb{T}$, which consists of the set of all $\lambda \in \mathbb{R}$ such that $\mathbb{T} - \lambda \mathbb{I}$ is not bijective. We recall some basic facts about the spectrum (see \textit{e.g.} \cite{reed2012methods}). The spectrum is always non-empty, closed and bounded. We define the \textit{discrete spectrum} $\sigma_{\mathrm{d}}(\mathbb{T})$ to be the part of $\sigma(\mathbb{T})$ which consists of all isolated points with finite algebraic multiplicity, which means that if $\lambda \in \sigma_d(\mathbb{T})$ then $\lambda$ is an eigenvalue of $\mathbb{T}$, that is there is some non-zero $f\in L^2([0,1])$ such that $\mathbb{T}f = \lambda f$, moreover $\cup_{r=1}^{\infty} \ker((\mathbb{T} - \lambda \mathbb{I})^r)$ is finite dimensional. We also define the \textit{essential spectrum} $\sigma_{\mathrm{ess}}(\mathbb{T}) = \sigma(\mathbb{T})\setminus\sigma_{\mathrm{d}}(\mathbb{T})$. According to \cite{edmunds1972non} the essential spectrum $\sigma_{\mathrm{ess}}(\mathbb{T})$ is the set of all numbers $\lambda$ in the spectrum $\sigma(\mathbb{T})$ of $\mathbb{T}$ such that at least one of the following conditions holds:
\begin{enumerate}[(i)]
    \item the range $R(\mathbb{T} - \lambda \mathbb{I})$ of $\mathbb{T} - \lambda \mathbb{I}$ is not closed;
    \item $\lambda$ is a limit point of $\sigma(\mathbb{T})$;
    \item $\cup_{r=1}^{\infty} \ker((\mathbb{T} - \lambda \mathbb{I})^r)$ is infinite dimensional.
\end{enumerate}

The essential spectrum is always closed, and the discrete spectrum can only have accumulation points on the boundary to the essential spectrum. For more results on the spectrum of linear operator we suggest \cite{reed2012methods}.  Denote by $\lambda_{\max}(\mathbb{T})$ the spectral radius of $\mathbb{T}$,
\begin{equation*}
    \lambda_{\max}(\mathbb{T}) = \sup\left\{\abs{\lambda} : \lambda \in \sigma(\mathbb{T})\right\}.
\end{equation*}
Note that  $\lambda_{\max}(\mathbb{T})\leq \normoo{\mathbb{T}}$. 

\subsection{Spectral Theory of the Graphon-Laplacian}

We define the \textit{graphon-Laplacian operator} $\mathbb{L}$ for a given graphon $w$ as

\begin{align}
    (\mathbb{L}f)(x) &:=\int_0^1w(x,y)(f(x)-f(y))dy \label{eq:LaplaceOperator}\\
    &= f(x)d(x)-\int_0^1 w(x,y)f(y)dy, \nonumber
\end{align}
where 
\begin{equation*}
     d(x)=\int_0^1w(x,y)dy
\end{equation*}

is the degree of $x$. The operator $\mathbb{L}$ for the graphon $w$ is the analogous of the Laplacian matrix of a graph. Similarly, let $\mathbb{L}_N$ and $\widehat{\mathbb{L}}_N$ denote the Laplacian operators associated to graphons $w_N$ and $\widehat{w}_N$ respectively, as defined in Equation~(\ref{eq:LaplaceOperator}).

We make some remarks on the spectrum of $\mathbb{L}$. Note that the operator $(\mathbb{W}f)(x) : = \int_0^1 w(x,y)f(y)dy$ is a Hilbert-Schmidt operator (see e.g. \cite{reed2012methods}), thus $\mathbb{W}$ is compact. Nevertheless, $\mathbb{L}$ is generally not compact since it is the sum of the compact operator  $\mathbb{W}$ and the multiplication operator $f(x) \mapsto d(x)f(x)$. It is well known that adding a compact operator does not change the essential spectrum (see e.g., \cite{gustafson1969essential}),  so the essential spectrum of $\mathbb{L}$ is the same as the essential spectrum of the multiplication operator $f(x) \mapsto d(x)f(x)$. The essential spectrum of this operator is known to be the range of the degree function $d$, i.e., $\sigma_{\mathrm{ess}}(\mathbb{L}) = rg(d)$ (see, \textit{e.g.}, \cite{von2008consistency} ). Note that the linear operator $\mathbb{L}$ is self-adjoint, bounded and positive.

It is easy to see that $\mathbb{L}e=0$, which means that $\lambda_1=0$ is an eigenvalue of $\mathbb{L}$. In fact it is the smallest eigenvalue, as:
\begin{equation}\label{eq:PositiveLaplacian}
    \langle \mathbb{L}f, f \rangle = \frac{1}{2}\int_0^1\int_0^1 w(x,y)(f(x)-f(y))^2dxdy\geq 0.
\end{equation}

If $\mathbb{L}f = 0$ and $w(x,y)$ is strictly positive for almost all $x,y$, then Inequality~(\ref{eq:PositiveLaplacian})  implies that $f\in \mathrm{Vect}(e)$.\footnote{We note here that any Robinsonian graphon $w$ that satisfies Assumption \ref{assumption:PartialDerivativeGraphon} will satisfy $w(x,y) > 0$ for all $x,y$. For that reason, the calculations in this paragraph will hold wherever Assumption \ref{assumption:PartialDerivativeGraphon} holds.} Thus, $\lambda_1=0$ is a  simple eigenvalue. Denote $\lambda_1=0 < \lambda_2 < \lambda_3 < \cdots$ the eigenvalues of $\mathbb{L}$ in increasing order (i.e., $\sigma_{\mathrm{d}}(\mathbb{L}) = \{ \lambda_1 = 0, \lambda_2, \lambda_3, \ldots\}$).  We call the \textit{Fiedler value} the smallest non-zero eigenvalue (i.e., $\lambda_2$) and any corresponding eigenfunction $\phi$ a \textit{Fiedler function}.

%Finally, we recall Kendall's notion of distance from the identity permutation: 
%\begin{equation*}
%    D_N(\sigma) = \sum_{i<j} \mathbf{1}_{\{\sigma(i)>\sigma(j)\}}.
%\end{equation*}
% Since it is not possible for us to distinguish between a permutation and its reverse, we define the symmetrized pseudometric
%\[
%D_{N}'(\sigma) = \min(D_{N}(\sigma), D_{N}(\sigma_{\mathrm{rev}})),
%\]

%where the \textit{reversed} permutation $\sigma_{\mathrm{rev}}$ is given by 
%\[
%\sigma_{\mathrm{rev}}(i) = N+1 - \sigma(i).
%\]

%Similarly, define 
%\[
%D_{N,\infty}'(\sigma) = \min(\sup_{i \in [N]} |\sigma(i) - i|, \sup_{i \in [N]} |\sigma_{\mathrm{rev}}(i) - i|).
%\]

\subsection{Statement of Main Results}

The main results relies on the following conditions on the graphon:

\begin{assumption}[Lipschitz graphon]\label{assumption:lipschitz} There exists a constant $K>0$ such that 
$$
\abs{w(x,y)-w(x',y')}\leq K (\abs{x-x'}+\abs{y-y'})
$$
for all pairs $(x,y)\in [0,1]^2, (x',y')\in [0,1]^2$.
\end{assumption}

\begin{assumption}[Differentiable graphon]\label{assumption:PartialDerivativeGraphon} The partial derivative  $\partial w(x,\cdot)/\partial x$ both exists and is non-zero almost everywhere.
\end{assumption}

\begin{remark} \label{RemToyGraph}
We note that Assumption \ref{assumption:PartialDerivativeGraphon} is violated by the popular toy graphon model $w(x,y) = p \mathbb{I}_{|x-y| \leq 0.5}$, and so Theorem \ref{ThmMainThm} will not apply as-written to these models. This is somewhat unsatisfying, as this model and very similar ones were studied in \cite{rocha2018recovering,7924316,RePEc:inm:oropre:v:68:y:2020:i:1:p:53-70}. 

Fortunately, this problem is easy to fix, and our arguments were written with this in mind. The proof of Theorem \ref{ThmMainThm} only uses Assumption \ref{assumption:PartialDerivativeGraphon} inside of the proofs of the intermediate results Theorems \ref{thm:monotone} and \ref{thm:SimpleFiedlerValue}. Thus, both the conclusion and the proof of Theorem \ref{ThmMainThm} remain correct if we assume the conclusions of Theorems \ref{thm:monotone} and \ref{thm:SimpleFiedlerValue} directly. 

More precisely, Theorem \ref{ThmMainThm} remains true as stated if we substitute the following pair of assumptions for Assumption \ref{assumption:PartialDerivativeGraphon}:

\begin{enumerate}
    \item There exists a Fiedler function $\phi\in\mathcal{C}^1([0,1])$ and $L > 0$ such that $\inf_{x \in [0,1]} \phi'(x) = L$ (this corresponds to Theorem \ref{thm:monotone}).
    \item The eigenspace corresponding to eigenvalue $\lambda_{2}$ is one-dimensional (this corresponds to Theorem \ref{thm:SimpleFiedlerValue}).
\end{enumerate}

These two assumptions can be checked directly for the toy model $w(x,y) = p \mathbb{I}_{|x-y| \leq 0.5}$, and we expect this to be possible for other very symmetric models. Due to perturbation bounds for the spectrum of operators, it is also possible to check this condition for sufficiently small perturbations of such symmetric models. 
\end{remark}

\begin{assumption}\label{assumption:NonZerod'}
The derivative $d'(x)$ exists for all $x \in [0,1]$ and the set $N := \left\{ x \in [0,1] : d'(x) = 0 \right\}$ is countable. 
\end{assumption}

\begin{assumption}\label{assumption:InfimumInequality}
Let 
\begin{equation*}
    \mathcal{H} = \left\{f\in L^2([0,1]) : \int_0^1 f(t)dt = 0 \text{ and } \int_0^1 f^2(t)dt = 1\right\},
\end{equation*}
and $\mathbb{L}$ be the graphon-Laplacian operator associated with $w$. Assume that 
\begin{equation*}
    \inf_{f\in \mathcal{H}} \langle \mathbb{L}f, f \rangle < \inf_{x\in [0,1]} d(x).
\end{equation*}
\end{assumption}

Note that $\inf (\sigma_{\mathrm{ess}}(\mathbb{L})) = \inf_{x \in [0,1]} d(x)$. Applying this with the mini-max principle for self-adjoint operators (see \cite{teschl2009mathematical}), Assumption~\ref{assumption:InfimumInequality}  ensures that the Fiedler value satisfies $\lambda_2 < \inf_{x \in [0,1]} d(x)$.

Although these assumptions may seem complicated, especially Assumption \ref{assumption:InfimumInequality}, all ``nice'' graphons in the sense of Definition \ref{Assumption:weakfamily} satisfy these assumptions. For the proof of this result, see Section \ref{SecApplNice}.

Our main result is the following consistency result for spectral seriation algorithm:

\begin{thm} \label{ThmMainThm}
Fix $\tau \in [0,1)$, $\eta<\frac{1-\tau}{3}$ and $\varepsilon\in (0,1)$. Let $\rho_{N} = N^{-\tau}$. Fix a graphon $w$ satisfying Assumptions \ref{assumption:lipschitz} to \ref{assumption:InfimumInequality}. Consider a family of random graphs  $\{G^{(N)}\}_{N \in \mathbb{N}}$ where $G^{(N)} \sim \rho_N w$ according to Equation~\eqref{EqDefSamplingRule}. Let $\hat{\sigma}^{(N)}$ be the output of Algorithm \ref{Alg:SpectralSeriation} with input $G^{(N)}$. Then there exists a positive constants $C = C(w)$ and an integer $N_{0} = N_{0}(w, \tau,\eta, \varepsilon)$ so that 
\begin{equation}\label{eq:approx1}
    \Prob{ \norm{\hat{\sigma}^{(N)} - id_N}_1 \leq C N^{2-\eta}} \geq 1 - \varepsilon
\end{equation}
for all $N \geq N_{0}$. 

Moreover, for any $\gamma > 1$ there exists a positive constants $C' = C'(w)$ and $N'_0 = N'_0(w, \tau, \eta, \gamma)$ so that 
\begin{equation}\label{eq:approx2}
    \Prob{ \norm{\hat{\sigma}^{(N)} - id_N}_1 \leq C' N^{2-\eta}} \geq 1 - \varepsilon_{\gamma}(N)
\end{equation}
for all $N \geq N'_0$.
\end{thm}

 Note that Inequality~\eqref{eq:approx2} looks very similar to Inequality~\eqref{eq:approx1} if we choose $\varepsilon = \varepsilon_{\gamma}(N)$. Although the two are very similar, the result in~\eqref{eq:approx2} can not be obtained as a consequence of~\eqref{eq:approx1}. In fact, the values of $N$ for which the first inequality is satisfied depend on $N_0$ which itself depend on $\varepsilon$.

With appropriate post-processing and some additional assumptions, we can achieve tighter estimates. For fixed $0 < \alpha < 0.5$, define the functions $\Psi_{R}, \Psi_{L}$ by
\begin{equation} \label{EqDefBigPsi}
    \Psi_{R}(x) := \int_{1-\alpha}^{1} w(x,y) dy \, \text{ and } \, \Psi_{L}(x) := \int_{0}^{\alpha} w(x,y) dy.
\end{equation}

These functions count the expected number of neighbours in the right- and left-most $\alpha$ percent of vertices. Assume:

\begin{assumption} \label{AssumptionPsiOK}

There exists some constants $0<\alpha < \beta <1/2$ such that:

\begin{enumerate}
    \item \textbf{Mean distance inequality:}
    \begin{equation}\label{IneqMDI}
        \inf_{x \in [1-\alpha,1]} \Psi_{R}(x) > \Psi_{R}(1-\beta) , \qquad \inf_{x \in [0,\alpha]} \Psi_{L}(x) > \Psi_{L}(\beta),
    \end{equation}
    and 
    \item \textbf{Distinguishability inequality:}  There exists some $d_1 > 0$ such that

    \begin{equation}\label{ineq:weaker-assumption}
       \{ \min(|x-z|, |y-z|) \geq \frac{\beta-\alpha}{2}\} \Rightarrow \{ |w(y,z)-w(x,z)| \geq d_{1} |x-y| \}
    \end{equation}
    for $x,y,z\in [0,1]$.
\end{enumerate}
\end{assumption}

This assumption holds also for all graphons that are ``nice'' in the sense of Definition \ref{Assumption:weakfamily}.  For the proof of this result, see Section \ref{SecApplNice}. When a graphon satisfies Inequality \eqref{IneqMDI} for \textit{some} pair $(\alpha,\beta)$ and also satisfies our other assumptions, it is possible to quickly find the pair $(\alpha,\beta)$; see Remark \ref{RemFindingAlphaBeta}.

The following theorem gives a much faster convergence rate under Assumption~\ref{AssumptionPsiOK}:

\begin{thm} \label{ThmSomePostprocMain}
Let $\rho_{N} = 1$, fix a graphon $w$ satisfying Assumptions  \ref{assumption:lipschitz} to \ref{AssumptionPsiOK}. Consider a family of random graphs  $\{G^{(N)}\}_{N \in \mathbb{N}}$ where $G^{(N)} \sim w$ according to Equation~\eqref{EqDefSamplingRule}. Let $\hat{\sigma}^{(N)}$ be the output of Algorithm \ref{AlgPostProc} with input $G^{(N)}$ and parameters $\alpha, \beta$ as in assumption~\ref{AssumptionPsiOK}.  Then for any $\gamma > 1$ and $\epsilon > 0$ there exists a positive constant $\delta = \delta(w,\gamma, \epsilon)$ and an integer $N_0 = N_0(w, \gamma, \epsilon)$ so that 
\begin{equation*}
    \Prob{ \norm{\hat{\sigma}^{(N)} - id_N}_{\infty} \leq \epsilon \sqrt{N\log^{\gamma}(N)}} \geq 1 - \varepsilon_{\gamma}^{\delta}(N)
\end{equation*}
for all $N \geq N_0$.
\end{thm}

\begin{remark}
Although we give specific functions $\Psi_{R}, \Psi_{L}$ in Equation \eqref{EqDefBigPsi}, our proofs only rely on three properties of this pair of functions: they (i) satisfy Assumption \ref{AssumptionPsiOK}, (ii) have derivatives bounded away from 0 on $[0,1-\beta]$ and $[\beta,1]$ respectively, and (iii) can be estimated with error roughly $O\left(\sqrt{\frac{\log(N)}{N}}\right)$ from a graph of size $N$. We state our results in terms of the specific functions $\Psi_{R},\Psi_{L}$ because it is easy to (i) compute them and (ii) verify our assumptions for the most popular graphon families. 
\end{remark}

\section{Properties of the Fiedler value and the Fiedler function} \label{SecFiedlerProp}

In this section we establish basic facts about the spectrum of the graphon-Laplacian operator $\mathbb{L}$. Throughout this section, we fix a graphon $w$ satisfying Assumptions~\ref{assumption:lipschitz} to \ref{assumption:InfimumInequality}, then define $\mathbb{L}$ as in Equation \eqref{eq:LaplaceOperator}. We also denote by $\lambda_{2}$ its Fiedler eigenvalue. For $x \in \mathbb{R}$, define $D_{x} = d^{-1}(\{x\})$. The following lemma ensures that most of the eigenfunctions of $\mathbb{L}$ are $\mathcal{C}^1$.
\begin{lemma}\label{thm:SmoothEigenfunction} Let $\lambda$ be any eigenvalue of $\mathbb{L}$ and $f$ any eigenfunction associated to the eigenvalue $\lambda$. Then the set $D_{\lambda}$ is discrete and $f$ is $\mathcal{C}^1$ in $D_{\lambda}^c$. 

In particular, if $\lambda\notin rg(d)$, then $f$ is $\mathcal{C}^1$ in $[0,1]$. 
\end{lemma}

\begin{proof}
 First we show that under Assumption~\ref{assumption:NonZerod'} the set $D_{\lambda}$ is discrete. Let $N = \{x\in [0,1] : d'(x)=0\}$ as in Assumption \ref{assumption:NonZerod'}. Since $N$ is discrete, it is enough to show that $D'=D_{\lambda}\setminus N$ is discrete. Let $x\in D'$, so that $d(x)=\lambda$ and $d'(x)\neq 0$. Moreover, $d(x+h)-\lambda= d(x+h)-d(x) = d'(x)h+o(h)$, hence $d(y)\neq \lambda$ for all $y\neq x$ in some neighborhood $U_x$ of $x$. Consequently every point $x\in D'$ has a neighborhood $U_x$ such that $D'\cap U_x = \{x\}$ as required.
 
 Now take $x \in D_{\lambda}^c$ and $\Delta >0$ such that $x+\delta \notin D_{\lambda}$ for all $\delta \in [0,\Delta]$. Recalling the definition of the operator $\mathbb{L}$ from Equation (\ref{eq:LaplaceOperator}),

 \begin{equation*}
     f(x)d(x) -  \int_0^1 f(y)w(x,y)dy = \lambda f(x).
 \end{equation*}
 Plugging $x$ and $x+\delta$ in the previous equation and taking the difference yields, for $\delta \in [0,\Delta]$ sufficiently small, 
\begin{align*}
    \frac{f(x+\delta)-f(x)}{\delta} &= \frac{1}{d(x+\delta)-\lambda}\left(\int_0^1(f(y)-f(x))\left(\frac{w(x+\delta,y)-w(x,y)}{\delta} \right)dy\right)\\
    &\xrightarrow[\delta \to 0]{}~\frac{1}{d(x)-\lambda}\left(\int_0^1 (f(y)-f(x))\frac{\partial w(x,y)}{\partial x}\right)dy.
\end{align*}
For the integral convergence, we apply the dominated convergence theorem. This is justified by the fact that, by Assumption~\ref{assumption:PartialDerivativeGraphon}, the partial derivative exists everywhere, and by Assumption ~\ref{assumption:lipschitz}, we have that $(w(x+\delta,y)-w(x,y))/\delta$ is uniformly bounded by $K < \infty$. Therefore $f$ is differentiable in $D_{\lambda}^c$ and for all $x\in D_{\lambda}^c$
\begin{equation}\label{eq:EigenfunctionDerivative}
    f'(x) = \frac{1}{d(x)-\lambda}\left(\int_0^1 (f(y)-f(x))\frac{\partial w(x,y)}{\partial x}\right)dy.
\end{equation}
By a similar argument we can show that $f'(x+\delta)\to f'(x)$ as $\delta\to 0$, consequently $f$ is $\mathcal{C}^1$ on $D_{\lambda}^c$. 

\end{proof}

Now we introduce the two theorems whose proofs will take the rest of the section: 

\begin{thm}\label{thm:monotone}
There exists a Fiedler function $\phi\in\mathcal{C}^1([0,1])$ and $L > 0$ such that $\inf_{x \in [0,1]} \phi'(x) = L$. 
\end{thm}

\begin{thm}\label{thm:SimpleFiedlerValue}
The eigenspace corresponding to eigenvalue $\lambda_{2}$ is one-dimensional.
\end{thm}

To prove Theorem~\ref{thm:monotone}, we introduce the linear operator
 \begin{equation}\label{eq:operatorM}
     (\mathbb{M}f)(x) =  d(x)f(x)-\int_0^1\int_0^1 \left( f(z)\frac{\partial w(x,y)}{\partial x}\mathbf{1}_{\{x\leq z\leq y\}}\right)dzdy.
 \end{equation}

Note that similarly to $\mathbb{L}$, the operator $\mathbb{M}$ is a self-adjoint, positive and bounded operator. Moreover, it is the sum of a multiplication operator and a compact operator, which means that $\sigma_{\mathrm{ess}}(\mathbb{M}) = rg(d)$. We will show that its spectral data is closely related to that of $\mathbb{L}$:

\begin{lemma}\label{lemma:operatorM} 

The eigenvalues of  $\mathbb{M}$ are $\lambda_2 < \lambda_3 < \cdots$, that is the spectrum of $\mathbb{M}$ satisfies  $\sigma(\mathbb{M})=\sigma(\mathbb{L})\setminus\{0\}$. Moreover, for all $\lambda \neq 0$, $(\lambda,f)$ is an eigenvalue-eigenfunction pair of $\mathbb{L}$ if and only if there exists a function $g$ so that $f'(x) = g(x)$ for all $x\in D_{\lambda}^c$  and $(\lambda,g)$ is an eigenvalue-eigenfunction pair of $\mathbb{M}$. \footnote{We have seen already that $D_{\lambda}$ is countable so in fact $f'$ exists almost everywhere. This means that the values of $f'$ where they exist will uniquely determine the element of $L^{2}([0,1])$ that is an eigenvalue-eigenfunction pair of $\mathbb{M}$.}
\end{lemma}

\begin{proof}
Let $(\lambda,f)$ be an eigenvalue-eigenfunction pair of $\mathbb{L}$ for which $\lambda\neq 0$. By Lemma ~\ref{thm:SmoothEigenfunction}, $f$ is differentiable at  all $x\in D_{\lambda}^c$, and by Assumption~\ref{assumption:PartialDerivativeGraphon}, the partial derivatives of $w$ exist almost everywhere. This justifies the calculation 
\begin{align*}
    (\mathbb{L}f)(x) = \lambda f(x) &\Rightarrow \lambda f(x) = f(x)d(x)-\int_0^1 f(y)w(x,y)dy\\
    &\Rightarrow \lambda f'(x) = f'(x)d(x)+f(x)d'(x)-\int_0^1 f(y)\frac{\partial w(x,y)}{\partial x} dy\\
    &\Rightarrow \lambda f'(x) = f'(x)d(x)-\int_0^1 (f(y)-f(x))\frac{\partial w(x,y)}{\partial x} dy\\
    &\Rightarrow \lambda f'(x) = f'(x)d(x)-\int_0^1\int_0^1 \left( f'(z)\frac{\partial w(x,y)}{\partial x}\mathbf{1}_{\{x\leq z\leq y\}}\right)dzdy\\
    &\Rightarrow \lambda f'(x) = (\mathbb{M}f')(x).\\
\end{align*}

Conversely, let $(\lambda,g)$ be an  eigenvalue-eigenfunction of $\mathbb{M}$. Then for all $x \notin D_{\lambda}$, straightforward algebra gives 
\begin{equation*}
    g(x)=\frac{1}{d(x)-\lambda}\int_0^1\int_0^1\left(\frac{\partial w(x,y)}{\partial x}g(z)\mathbf{1}_{\{x\leq z\leq y\}}\right)dzdy.
\end{equation*}

Applying the dominated convergence theorem (this can be used because the partial derivative of $w$ satisfies $\left|\frac{\partial w(x,y)}{\partial x} \right| \leq K$ when it exists, by Assumption ~\ref{assumption:lipschitz}), we can see that $g$ is continuous on $D_{\lambda}^c$.

Define $f(x)=\int_{0}^{x} g(t)dt$ for all $x\in[0,1]$. Then $f'(x)=g(x)$ for $x\in D_{\lambda}^c$, and furthermore
\begin{align*}
    (\mathbb{M}g)(x) = \lambda g(x) &\Rightarrow d(x)f'(x)+\int_0^1 (f(x)-f(y))\frac{\partial w(x,y)}{\partial x}dy = \lambda f'(x)\\
    &\Rightarrow \int_0^1 \left(\frac{\partial (f(x)-f(y))}{\partial x}w(x,y)+(f(x)-f(y))\frac{\partial w(x,y)}{\partial x}\right)dy = \lambda f'(x)\\
    &\Rightarrow \int_0^1 \frac{\partial }{\partial x}\left((f(x)-f(y))w(x,y)\right)dy = \lambda f'(x)\\
    &\Rightarrow \int_0^1 (f(x)-f(y))w(x,y)dy = \lambda f(x)\\
    &\Rightarrow (\mathbb{L}f)(x) = \lambda f(x)\\
\end{align*}
as required. 

Finally, we check that $0 \notin \sigma(\mathbb{M})$. To see this, consider a function $g$ such that $\mathbb{M}g=0$. Then by the above calculation $\mathbb{L}f=0$, where $f(x) = \int_{0}^{x} g(t) dt$. Recalling from the paragraph following \eqref{eq:PositiveLaplacian} that $\lambda_{1}$ is a simple eigenvalue of $\mathbb{L}$, this implies $f\in\mathrm{Vect}(e)$ and so $g=f'=0$. We conclude that $\sigma(\mathbb{M})=\sigma(\mathbb{L})\setminus\{0\}$.
\end{proof}

\begin{lemma}\label{lemma:PositiveEigenfunction}
There exists $f^*\in L^2([0,1])$ such that $\mathbb{M}f^* = \lambda_2 f^*$ and $f^*(x)>0$ for all $x\in[0,1]$.
\end{lemma}

\begin{proof}
 Fix $\alpha > 1+\lambda_{\max}(\mathbb{M})$.  Define the operator $\mathbb{M}_{\alpha} = \alpha \mathbb{I}-\mathbb{M}$ with $\mathbb{I}$ the identity operator. It is obvious that $\sigma(\mathbb{M}_{\alpha}) = \alpha - \sigma(\mathbb{M})$, hence $\lambda_{\max}(\mathbb{M}_{\alpha})>0$. 
 
 We next claim that  $\mathbb{M}_{\alpha}$ sends the cone of nonnegative functions,
$$
K=\left\{f\in L^2([0,1]) :  f(x)\geq 0 \ \text{for all} \ x\in[0,1]\right\},
$$
to itself. To see this, we examine the following expression for $f \in K$ and $x \in [0,1]$:
\begin{align*}
    (\mathbb{M}_{\alpha}f)(x) = (\alpha-d(x))f(x)+\int_0^1 \left(\int_x^y f(z)dz\right)\frac{\partial w(x,y)}{\partial x}dy.
\end{align*}
The first term $(\alpha-d(x))f(x)$ is non-negative since $\alpha>1\geq d(x)$. For the second term, if $y>x$ then $\int_x^y f(z)dz\geq 0$ and  $\frac{\partial w(x,y)}{\partial x}\geq 0$ (because $w$ is Robinsonian). Similarly, if $y<x$, both of these are negative and so their product is positive. Hence, in either case, $\left(\int_x^y f(z)dz\right)\frac{\partial w(x,y)}{\partial x}\geq 0$ for all $y$. Therefore $\mathbb{M}_{\alpha}f \in K$, as required.

Consequently $\mathbb{M}_{\alpha}$ is a positive, bounded and self-adjoint linear operator, with $\mathbb{M}_{\alpha} K \subseteq K$.  The essential spectrum of $\mathbb{M}_{\alpha}$ is $\sigma_{\mathrm{ess}}(\mathbb{M}_{\alpha})=\alpha - rg(d)$ (because $\sigma_{\mathrm{ess}}(\mathbb{M})=rg(d)$). Since $\lambda_{\max}(\mathbb{M}_{\alpha})$ is the largest eigenvalue of $\mathbb{M}_{\alpha}$ and $\sigma(\mathbb{M}_{\alpha})=\alpha-\sigma(\mathbb{M})$, additionally $\sigma(\mathbb{M}) = \sigma(\mathbb{L})\setminus \{0\}$ by Lemma~\ref{lemma:operatorM} and so  $\lambda_{\max}(\mathbb{M}_{\alpha}) = \alpha - \lambda_2$.

In \cite{edmunds1972non} the authors showed a version of the Krein-Rutman theorem for non-compact operators, which applies here. In fact according to Theorem 1 from \cite{edmunds1972non}, there exists $f^*\in K\setminus\{0\}$ such that $\mathbb{M}_{\alpha}f^* = \lambda_{\max}(\mathbb{M}_{\alpha})f^*$, note also that the conditions of this theorem are satisfied since $\lambda_{\max}(\mathbb{M}_{\alpha}) = \alpha - \lambda_2 > \alpha - \min_{x \in [0,1]}d(x) = \sup \sigma_{\mathrm{ess}}(\mathbb{M}_{\alpha})$ according to Assumption~\ref{assumption:InfimumInequality}. Note this is equivalent to $\mathbb{M} f^* = \lambda_2 f^*$ as required.

 We must now check that $f^{*}(x) > 0$ for all $x \in [0,1]$. Assume not; then there exists some $a \in [0,1]$ for which $f^*(a)=0$. We then have $(\mathbb{M}f^*)(a)=\lambda_2 f^*(a)=0$, which yields
\begin{equation} \label{EqOneBadPointImpliesZero}
    \int_0^1 \left(\int_a^y f^*(z)dz\right)\frac{\partial w(a,y)}{\partial x}dy=0.
\end{equation}
Thus, for almost all $y\in[0,1]$,
\begin{equation} \label{EqIntProd0}
    \left(\int_a^y f^*(z)dz\right)\frac{\partial w(a,y)}{\partial x}  = 0.
\end{equation}
By Assumption~\ref{assumption:PartialDerivativeGraphon}, $\frac{\partial w(a,y)}{\partial x} \neq 0$ almost everywhere, so Equation~\eqref{EqIntProd0} implies $\int_a^y f^*(z)dz =0$ a.e. This means that $f^*=0$ in $L^2([0,1])$, which contradicts the fact that $f^{*}$ is an eigenfunction of $\mathbb{M}$. We conclude that no such $a \in [0,1]$ exists, so $f^*(x)>0$ for all $x\in[0,1]$ as desired.
\end{proof}

\begin{corollary}\label{cor:notinrange}
The Fiedler value $\lambda_2\notin rg(d)$.
\end{corollary}

\begin{proof}
 Assume not, and let $\lambda_2 = d(a)$ for some $a$. Let $f^{*}$ be as in the statement of Lemma \ref{lemma:PositiveEigenfunction}. Since $\mathbb{M}f^*(a)= \lambda_2 f^*(a)$, we get 
 \begin{equation*}
    \int_0^1 \left(\int_a^y f^*(z)dz\right)\frac{\partial w(a,y)}{\partial x}dy=0.
\end{equation*}
Following the argument in the proof of Lemma \ref{lemma:PositiveEigenfunction} from Equation \eqref{EqOneBadPointImpliesZero}, this implies $f^*=0$ in $L^2([0,1])$. But we have already seen that $f^{*}(x) > 0$ for all $x \in [0,1]$, so this is a contradiction.
\end{proof}

We can now prove Theorem~\ref{thm:monotone}:

\begin{proof}[Proof of Theorem~\ref{thm:monotone}]
By Corollary~\ref{cor:notinrange},  $\lambda_2\notin rg(d)$, and so  by Lemma~\ref{thm:SmoothEigenfunction} all eigenfunctions of $\mathbb{L}$ associated with $\lambda_2$ are $\mathcal{C}^1([0,1])$. Let $f^*$ be as in the statement of Lemma~\ref{lemma:PositiveEigenfunction}, and define $\phi(x)=\int_0^x f^*(t)\mathrm{d}t$ for all $x\in[0,1]$. By Lemma~\ref{lemma:operatorM} we know that $\mathbb{L}\phi=\lambda_2\phi$, thus $\phi\in\mathcal{C}^1([0,1])$ is a Fiedler function with $\phi'(x)=f^*(x)>0$ for all $x\in [0,1]$.

Finally, since $\phi \in \mathcal{C}^{1}([0,1])$ and $\phi'(x) > 0$ for all $x \in [0,1]$, we have $\inf_{x \in [0,1]} \phi'(x) > 0$.
\end{proof}

Finally, we present the proof of Theorem~\ref{thm:SimpleFiedlerValue}:

\begin{proof}[Proof of Theorem~\ref{thm:SimpleFiedlerValue}]
Let $\phi$ as in the statement of Theorem~\ref{thm:monotone}. We prove Theorem~\ref{thm:SimpleFiedlerValue} by contradiction. If the eigenspace is not one-dimensional, there exists some other Fiedler function $\varphi \notin \mathrm{Vect}(\phi)$. Since $\lambda_2\notin rg(d)$ by Corollary \ref{cor:notinrange}, we have that $\varphi\in\mathcal{C}^1([0,1])$ as well. Let $a$ be such that 
\begin{equation*}
    \frac{\varphi'(a)}{\phi'(a)} = \min_{x\in[0,1]}\left(\frac{\varphi'(x)}{\phi'(x)}\right).
\end{equation*}
Put $\psi = \phi'(a)\varphi - \varphi'(a)\phi$, so that $\psi$ is an other Fiedler function and satisfies $\psi'(x)\geq 0$ for all $x\in[0,1]$ and $\psi'(a)=0$. But this implies  $\mathbb{M}\psi'(a)=\lambda_2 \psi'(a)=0$. Following the argument in the proof of Lemma \ref{lemma:PositiveEigenfunction} from Equation \eqref{EqOneBadPointImpliesZero}, this implies that $\psi=0$ in $L^2([0,1])$, which is a contradiction.
\end{proof}
\section{Convergence of the Fiedler eigenfunction} \label{SecConv}

We follow the conventions established at the start of Section \ref{SecFiedlerProp}, fixing a graphon $w$ satisfying Assumptions~\ref{assumption:lipschitz} to \ref{assumption:InfimumInequality} and defining the Laplacian $\mathbb{L}$ as in Equation \eqref{eq:LaplaceOperator}. We will also use the notation for random graphs and their associated objects from Section \ref{SecNotation}. 

In this section, we will check that the observed graph Laplacian converges to the limiting Laplacian $\mathbb{L}$ in a strong enough sense to recover spectral data about the Fiedler eigenfunction.

\subsection{Convergence of Laplacian operator}

The results in this section are very similar to the results in the paper \cite{avella2018centrality}. In this section we prove that two approximations to $\mathbb{L}$ converge quickly. But first, we set some notations. Define $d_{\mathrm{MAX}} := \max_{x\in[0,1]}d(x)$ and similarly $d_{\mathrm{MAX}}^{(N)} := \max_{i=1}^N\left(\sum_{j=1}^N P_{ij}^{(N)}\right)$.  Let $K$ be as in Assumption ~\ref{assumption:lipschitz} and let $\gamma > 1$. Define the integers $N_1(w, \tau, \varepsilon)$ and $N'_1(w, \tau, \gamma)$ as follows:
\begin{align}
    N_1(w, \tau, \varepsilon) &= \min\left\{ n \in \mathbb{N} \, : \, N \geq n \Rightarrow d_{\mathrm{MAX}} \geq  \frac{4\log\left(4N/\varepsilon\right)}{9\rho_{N} N}+\frac{2K}{N} \right\}, \label{eq:SmallestIntegers1}\\
    N'_1(w, \tau, \gamma) &= \min\left\{ n \in \mathbb{N} \, : \, N \geq n \Rightarrow d_{\mathrm{MAX}} \geq  \frac{4\log\left(4N/\varepsilon_{\gamma}(N)\right)}{9\rho_{N} N}+\frac{2K}{N} \right\}. \label{eq:SmallestIntegers2}
\end{align}
Recall that $\varepsilon_{\gamma}(N) = \exp{\left(-\log^{\gamma}(N)\right)}$. Note that both $N_1$ and $N'_1$ are well defined. In fact the term $\frac{4\log\left(2N/\varepsilon\right)}{9\rho_{N} N}$ goes to 0 as $N \to \infty$ and we also have
\begin{equation*}
    \frac{4\log\left(4N/\varepsilon_{\gamma}(N)\right)}{9\rho_{N} N} = \frac{4\left( \log(4N) + \log^{\gamma}(N)\right)}{9\rho_N N}
\end{equation*}
which goes to 0 as $N \to \infty$.

In order to prove the convergence of the Laplacian operator we need the following simple technical lemma:
\begin{lemma}\label{lemma:MaxDegree}
Fix $\varepsilon\in (0,1)$, $\gamma > 1$. Put $N_1 = N_1(w, \tau, \varepsilon)$ and $N'_1 = N'_1(w, \tau, \gamma)$. Then for all $N \geq N_1(w, \tau, \varepsilon)$ we have
 \begin{equation*}
    d_{\mathrm{MAX}}^{(N)} \geq  \frac{4\log\left(4N/\varepsilon\right)}{9\rho_N}.
 \end{equation*}
Similarly, for all $N \geq N'_1$ we have
 \begin{equation*}
    d_{\mathrm{MAX}}^{(N)} \geq  \frac{4\log\left(4N/\varepsilon_{\gamma}(N)\right)}{9\rho_N}.
 \end{equation*}
 \end{lemma}

\begin{proof}
For $N \geq N_1$, let $x\in \iI_i^N$ for some $i \in [N]$. Then 
\begin{align*}
    \abs{d(x)-d_N(x)} & = \abs{\int_0^1\left(w(x,y)-w_N(x,y)\right)dy}\\
    & \leq \sum_{j=1}^N \int_{\iI_j^N}\abs{w(x,y)-w(\frac{i}{N},\frac{j}{N})}dy\\
    &\leq \frac{2K}{N},
\end{align*}
where the inequality from line two to line three follows from Assumption~\ref{assumption:lipschitz}. Therefore $d_N(x)\geq d(x)-\frac{2K}{N}$, so
\begin{align*}
    \max_{x\in [0,1]}d_N(x)&\geq d_{\mathrm{MAX}}-\frac{2K}{N}&\\
    &\geq \frac{4\log(4N/\varepsilon)}{9\rho_N N}.
\end{align*}
On the other hand, 
\begin{align*}
    \max_{i=1}^N\left(\sum_{j=1}^NP_{ij}^{(N)}\right) &= N\max_{x\in[0,1]}\left(d_N(x)\right)\\
    &\geq \frac{4\log(4N/\varepsilon)}{9\rho_N}.
\end{align*}
As required. The proof of the second inequality for $N \geq N'_1(w, \tau, \gamma)$ is identical to the previous proof (replace $\varepsilon$ by $\varepsilon_{\gamma}(N)$).
\end{proof}

\begin{theorem}[Convergence of graphon Laplacian]\label{thm:ConvLaplace}
Let $K$ be as in Assumption~\ref{assumption:lipschitz}. Then
\begin{equation}
    \normoo{\mathbb{L}_N-\mathbb{L}} \leq \frac{4K}{N}. \label{ineq:LaplaceConv}
\end{equation}

Moreover, for all $\varepsilon\in(0,1)$ and all $N \geq N_1(w, \tau, \varepsilon)$ as in Equation~\eqref{eq:SmallestIntegers1},
\begin{equation}\label{ineq:SampleLaplaceConv}
    \Prob{ \normoo{\rho_N^{-1}\widehat{\mathbb{L}}_N-\mathbb{L}}\leq 4\sqrt{\frac{\log(4N/\varepsilon)}{\rho_N N}}+\frac{4K}{N} } \geq 1 - \varepsilon.
\end{equation}
Additionally, for all $\gamma>1$ the same inequality holds for all $N \geq N'_1(w, \tau, \gamma)$ if we replace $\varepsilon$ by $\varepsilon_{\gamma}(N)$, where $N'_1$ is defined in~\eqref{eq:SmallestIntegers2}.
\end{theorem}

\begin{proof}[Proof of Theorem~\ref{thm:ConvLaplace}] We prove the two inequalities \eqref{ineq:LaplaceConv} and \eqref{ineq:SampleLaplaceConv} in order.

\begin{proof}[Proof of Inequality~(\ref{ineq:LaplaceConv})]

 Let $f\in L^2([0,1])$ such that $\norm{f}=1$. Then
 \begin{align}
     \norm{\mathbb{L}_Nf-\mathbb{L}f}^2 &= \int_0^1\left(\mathbb{L}_Nf-\mathbb{L}f\right)^2(x)dx \nonumber \\ 
     &= \int_0^1\left(\int_0^1(f(x)-f(y))(w_N(x,y)-w(x,y))dy\right)^2dx \nonumber \\
     &\leq \int_0^1\left(\int_0^1(f(x)-f(y))^2dy\right)\left(\int_0^1(w_N(x,y)-w(x,y))^2dy\right)dx \nonumber \\
     &= \sum_{i=1}^N\int_{\iI_i^N}\left(\int_0^1(f(x)-f(y))^2dy\right)\left(\int_0^1(w_N(x,y)-w(x,y))^2dy\right)dx, \label{ineq:normconv}
 \end{align}
where from line two to three we used Cauchy-Schwartz inequality. For fixed $i \in [N]$ and $x\in \iI_i^N$,
\begin{align*}
    \int_0^1(w_N(x,y)-w(x,y))^2dy &= \sum_{j=1}^N\int_{\iI_j^N}(w(\frac{i}{N},\frac{j}{N})-w(x,y))^2dy\\
    & \leq \sum_{j=1}^N \int_{\iI_j^N} K^2 \left(\abs{\frac{i}{N}-x}+\abs{\frac{j}{N}-y}\right)^2dy\\
    & \leq \sum_{j=1}^N\int_{\iI_j^N} \frac{4K^2}{N^2}dy\\
    & = \frac{4K^2}{N^2}.
\end{align*}
Furthermore,
\begin{align*}
    \int_0^1(f(x)-f(y))^2dy & \leq \int_0^1 2(f(x)^2+f(y)^2)dy\\
    &= 2f(x)^2+2.
\end{align*}
Applying the last two calculations with Inequality~(\ref{ineq:normconv}),
\begin{align*}
     \norm{\mathbb{L}_{Nf}-\mathbb{L}f}^2 &\leq \frac{8K^2}{N^2}  \sum_{i=1}^N\int_{\iI_i^N}\left(f(x)^2+1\right)dx\\
     &=\frac{16K^2}{N^2}.
\end{align*}
This immediately implies ~(\ref{ineq:LaplaceConv}).
\end{proof}

\begin{proof}[Proof of Inequality~(\ref{ineq:SampleLaplaceConv})]

From the triangle inequality and Inequality~(\ref{ineq:LaplaceConv}),
\begin{equation}\label{ineq:SampledLaplaceBound}
    \normoo{\rho_N^{-1}\widehat{\mathbb{L}}_N-\mathbb{L}}\leq \normoo{\rho_N^{-1}\widehat{\mathbb{L}}_N-\mathbb{L}_N}+\normoo{\mathbb{L}_N-\mathbb{L}}\leq \normoo{\rho_N^{-1}\widehat{\mathbb{L}}_N-\mathbb{L}_N}+\frac{4K}{N}.
\end{equation}

Since both $\widehat{\mathbb{L}}_N$ and $\mathbb{L}_N$ have finite rank, it is not hard to see that 
\begin{align}
    \normoo{\rho_N^{-1}\widehat{\mathbb{L}}_N-\mathbb{L}_N} &= \frac{1}{N}\lambda_{\mathrm{max}}\left(\rho_N^{-1}\widehat{L}^{(N)}-L^{(N)}\right) \nonumber \\
    &= \frac{\rho_N^{-1}}{N}\norm{\widehat{L}^{(N)}-\rho_N L^{(N)}} \nonumber  \\
    & \leq \frac{\rho_N^{-1}}{N}\left(\norm{\widehat{D}^{(N)}-\rho_N D^{(N)}}+\norm{\widehat{P}^{(N)}-\rho_N P^{(N)}}\right)  \label{ineq:TriangIneq}.
\end{align}

Fix $N \geq N_1(w, \tau, \varepsilon)$ as in \eqref{eq:SmallestIntegers1}, then the maximum expected degree of the random graph $G^{(N)}$ is $\rho_{N} d_{\mathrm{MAX}}^{(N)}\geq \frac{4}{9}\log(4N/\varepsilon)$. Applying Theorem 1 of \cite{chung2011spectra}, 
\begin{equation}\label{ineq:BoundAdjacency}
    \Prob{\norm{\widehat{P}^{(N)}-\rho_N P^{(N)}} \leq \sqrt{4\rho_N d_{\mathrm{MAX}}^{(N)}\log(4N/\varepsilon)}} \geq 1 - \frac{\varepsilon}{2}.
\end{equation}
Now we give a similar bound for $\norm{\widehat{D}^{(N)}-\rho_N D^{(N)}}$. Let $\Delta^{(N)}=\widehat{D}^{(N)}-\rho_N D^{(N)}$ and $X^{(N)}=\widehat{P}^{(N)}-\rho_N P^{(N)}$. 

We note that
\begin{equation} \label{IneqSimpleIneqForBern}
    \sum_j \Var{X_{ij}}=\sum_j \rho_N P_{ij}^{(N)}(1-\rho_N P_{ij}^{(N)})\leq \rho_N \sum_j P_{ij}^{(N)}\leq \rho_N  d_{\mathrm{MAX}}^{(N)}.
\end{equation}

Then by Bernstein's inequality, for each $t>0$
\begin{align*}
\Prob{|\Delta_i^{(N)}|>t} &= \Prob{\abs{\sum_{j=1}^NX_{ij}^{(N)}}>t}\\
&\leq 2\exp{\left(\frac{-\frac{1}{2}t^2}{\sum_j \Var{X_{ij}}+\frac{1}{3}t}\right)}\\
&\stackrel{\text{Ineq. (}\ref{IneqSimpleIneqForBern}\text{)}}{\leq} 2\exp{\left(\frac{-\frac{1}{2}t^2}{\rho_N d_{\mathrm{MAX}}^{(N)} +\frac{1}{3}t}\right)}.
\end{align*}

Taking a union bound,

\begin{equation*}
    \Prob{\norm{\Delta^{(N)}}>t} \leq 2N\exp{\left(\frac{-\frac{1}{2}t^2}{\rho_N d_{\mathrm{MAX}}^{(N)} +\frac{1}{3}t}\right)}.
\end{equation*}
Choosing $t=\sqrt{4\rho_N d_{\mathrm{MAX}}^{(N)}\log\left(4N/\varepsilon\right)}$,

\begin{align*}
    \Prob{\norm{\Delta^{(N)}}>\sqrt{4\rho_N d_{\mathrm{MAX}}^{(N)}\log\left(4N/\varepsilon\right)}} & \leq 2N\exp{\left(-\frac{2\rho_N d_{\mathrm{MAX}}^{(N)}\log(4N/\varepsilon)}{\rho_N  d_{\mathrm{MAX}}^{(N)} + \frac{1}{3}\sqrt{4\rho_N d_{\mathrm{MAX}}^{(N)}\log\left(4N/\varepsilon\right)}}\right)}\\
    & \leq 2N \exp{\left(-\frac{2\rho_N d_{\mathrm{MAX}}^{(N)}\log(4N/\varepsilon)}{\rho_N d_{\mathrm{MAX}}^{(N)}+\rho_N d_{\mathrm{MAX}}^{(N)}}\right)}\\
    & = \frac{\varepsilon}{2}.
\end{align*}
Recalling that $\norm{\Delta^{(N)}} = \norm{\widehat{D}^{(N)}-\rho_N D^{(N)}}$ and $d_{\mathrm{MAX}}^{(N)}\leq N$, this implies
\begin{equation}\label{ineq:BoundDelta}
    \Prob{\norm{\widehat{D}^{(N)}-\rho_N D^{(N)}}  \leq \sqrt{4\rho_N N\log\left(4N/\varepsilon\right)}} \geq 1 - \frac{\varepsilon}{2}.
\end{equation}
 For $\varepsilon\in (0,1)$ we substitute~(\ref{ineq:BoundAdjacency}) and~(\ref{ineq:BoundDelta}) in~(\ref{ineq:TriangIneq}) to get 
\begin{equation}
    \Prob{\normoo{\rho_N^{-1}\widehat{\mathbb{L}}_N-\mathbb{L}_N}  \leq 4\sqrt{\frac{\log(4N/\varepsilon)}{\rho_N N}}} \geq 1 - \varepsilon.
\end{equation}

Finally, applying inequality~(\ref{ineq:SampledLaplaceBound}), we get that for all $\varepsilon\in(0,1)$ and all $N \geq N_1(w, \tau, \varepsilon)$ in the sense of~\eqref{eq:SmallestIntegers1}
\begin{equation*}
    \Prob{\normoo{\rho_N^{-1} \widehat{\mathbb{L}}_N-\mathbb{L}}\leq 4\sqrt{\frac{\log(4N/\varepsilon)}{\rho_N N}}+\frac{4K}{N}} \geq 1 - \varepsilon,
\end{equation*}
completing the proof. Again, replace in the $\varepsilon$ by $\varepsilon_{\gamma}(N)$ in the above and get the same result for all $N \geq N'_1(w, \tau, \gamma)$ in the sense of~\eqref{eq:SmallestIntegers2}.
\end{proof}
\end{proof}

\subsection{Convergence of the Fiedler function}

In order to prove the convergence of the Fiedler function we use a result by Davis and Kahan \cite{davis1970rotation} known as the $\sin(\theta)$-theorem for self-adjoint operators.

\begin{thm}[Davis-Kahan]\label{thm:davis-kahan}
Let $\mathbb{A}$ and $\widetilde{\mathbb{A}}$ be two self-adjoint operators with eigenvalues $\lambda_1 < \lambda_2 < \lambda_3 < \cdots$ and 
$\tilde{\lambda}_1 < \tilde{\lambda}_2 <  \tilde{\lambda}_3 <\cdots$ respectively. Consider eigenfunctions $f$ and $\tilde{f}$ of $\mathbb{A}$ and $\widetilde{\mathbb{A}}$ respectively with eigenvalues $\lambda_2$ and $\tilde{\lambda}_2$ respectively, and normalized so that $\norm{f}=\|\tilde{f}\|=1$ and $\langle f,\tilde{f} \rangle \geq 0$. If $0<\lambda_2<\lambda_3$ (i.e., $\lambda_2$ is isolated), then
\begin{equation*}
    \norm{\tilde{f}-f} \leq \frac{2\sqrt{2}\normoo{\widetilde{\mathbb{A}}-\mathbb{A}}}{\min\{\lambda_2 - \lambda_1,\lambda_3-\lambda_2\}}.
\end{equation*}
\end{thm}

Applying this to $\mathbb{L}$ and $\rho_N^{-1} \widehat{\mathbb{L}}_N$ immediately gives:

\begin{thm}\label{thm:FiedlerConv}
Let $\phi$ and $\hat{\phi}_N$ be the Fiedler functions of $\mathbb{L}$ and $\rho_N \widehat{\mathbb{L}}_N$ respectively, normalized so that $\|\phi\|=\|\hat{\phi}_N\|=1$ and $\langle \phi, \hat{\phi}_N\rangle\geq 0$. Fix $\eta < \frac{1-\tau}{3}$, then for all $\varepsilon \in (0,1)$ there exists a positive constant $C = C(w)$ and an integer $N_0 = N_0(w, \tau, \eta, \varepsilon)$ so that
\begin{equation*}
    \Prob{\norm{\hat{\phi}_N - \phi} \leq C \sqrt{N^{-3\eta}}} \geq 1 - \varepsilon
\end{equation*}
for all $N \geq N_{0}$. 

Moreover, for any $\gamma > 1$ there exists a positive constant $C' = C'(w)$ and an integer $N'_0 = N'_0(w, \tau, \eta, \gamma)$ so that 
\begin{equation*}
    \Prob{\norm{\hat{\phi}_N-\phi} \leq C'\sqrt{N^{-3\eta}}} \geq 1 - \varepsilon_{\gamma}(N)
\end{equation*}
for all $N \geq N'_0$.
\end{thm}

\begin{proof}
Define $N_0 = N_0(w, \tau, \eta, \varepsilon)$ and $N'_0 = N'_0(w, \tau, \eta, \gamma)$ by 
\begin{align*}
    &N_0(w, \tau, \eta, \varepsilon) 
     = \min \left\{n \geq N_1(w, \tau, \varepsilon) \, : \,  N \geq n \Rightarrow 4\sqrt{\frac{\log(4N/\varepsilon)}{\rho_N N}}+\frac{4K}{N} \leq \sqrt{N^{-3\eta}} \right\} \\
    &N'_0(w, \tau, \eta, \gamma) = \min \left\{n \geq N'_1(w, \tau, \eta) \, : \,  N \geq n \Rightarrow 4\sqrt{\frac{\log(4N/\varepsilon_{\gamma}(N))}{\rho_N N}}+\frac{4K}{N} \leq \sqrt{N^{-3\eta}} \right\}
\end{align*}
where $N_1$ and $N'_1$ are defined in \eqref{eq:SmallestIntegers1}, \eqref{eq:SmallestIntegers2} respectively. Note that both $N_0$ and $N'_0$ are well defined since the left hand sides of the inequalities inside $N_0$ and $N'_0$ respectively are on the order of $\mathcal{O}(N^{-(1-\tau)/2}\log^{1/2}(N))$ and  $\mathcal{O}(N^{-(1-\tau)/2}\log^{\gamma/2}(N))$ respectively, which converges faster than the right hand sides of order $\mathcal{O}(N^{-3\eta/2})$ (recall that $\tfrac{3\eta}{2} < \tfrac{1-\tau}{2}$).

Take $N \geq N_0$, by Theorem~\ref{thm:ConvLaplace} we know that with a probability of at least  $1-\varepsilon$
\begin{align*}
    \normoo{\rho_N^{-1}\widehat{\mathbb{L}}_N-\mathbb{L}}
    & \leq 4\sqrt{\frac{\log(4N/\varepsilon)}{\rho_N N}}+\frac{4K}{N} \\
    & \leq \sqrt{N^{-3\eta}}.
\end{align*}
We apply the Davis-Kahan Theorem~\ref{thm:davis-kahan} to the Fiedler functions $\phi$ and $\hat{\phi}_N$ (note that the theorem applies here since under our assumptions $\lambda_2$ is simple by Theorem~\ref{thm:SimpleFiedlerValue}). We have with probability at least $1-\varepsilon$
\begin{align*}
     \norm{\hat{\phi}_N-\phi} 
     & \leq \frac{2\sqrt{2}\normoo{\rho_N^{-1}\widehat{\mathbb{L}}_N-\mathbb{L}}}{\min\{\lambda_2,\lambda_3-\lambda_2\}}\\
     & \leq \frac{2\sqrt{2}}{\min\{\lambda_2,\lambda_3-\lambda_2\}} \sqrt{N^{-3\eta}}
\end{align*}
 as required. We get the same result for the case $\varepsilon = \varepsilon_{\gamma}(N)$  and $N \geq N'_0$.
\end{proof}

\section{Error bound for the spectral seriation algorithm}\label{SecErrorBound}

In this section, we prove Theorem \ref{ThmMainThm}. The method we use to bound the errors of the permutation are very similar to the method used in the paper \cite{rocha2018recovering}. We follow the notation used in the statement of  Theorem \ref{ThmMainThm}, denote by $\phi$ the Fiedler eigenfunction of $\mathbb{L}$, and denote by $\hat{\phi}_N$ the Fiedler vector appearing in step 3 of Algorithm \ref{Alg:SpectralSeriation} with input $G^{(N)}$. By a small abuse of notation, for $i \in [N]$ we define $\hat{\phi}_i$ to be the $i$'th element of ${\phi}_N$. Similarly, for $x \in [0,1]$, we define 
\begin{equation*}
    \hat{\phi}_{N}(x) = \sum_{i = 1}^{N} \hat{\phi}_i 1_{\iI_i^N}(x).
\end{equation*}
It is straightforward to check that $\hat{\phi}_{N}$ is exactly the Fiedler vector of $\rho_N \widehat{\mathbb{L}}_N$. Note that we use the same notation $\hat{\phi}_{N}$ to refer to both the Fiedler vector associated with finite graph $G^{(N)}$ \textit{and} the Fiedler \textit{function} associated with the empirical graphon $\rho_N \widehat{\mathbb{L}}_N$. This will not cause any ambiguity in the following argument, as it is always clear if the argument for $\hat{\phi}_{N}$ is  a generic value $x \in [0,1]$ or a generic value $i \in \{1,2,\ldots,N\}$. We repeat notation in this way because the two objects take exactly the same values in the same order, and this allows us to avoid some repetition of essentially identical calculations.

Let $\sigma = \sigma_{\hat{\phi}_N}$ as in Equation~\eqref{Eq:DefFuncToPerm}. Note that $\sigma$ is the output of Algorithm~\ref{Alg:SpectralSeriation} with input $G^{(N)}$ and that we have $\hat{\phi}_{\sigma(1)}\leq \hat{\phi}_{\sigma(2)}\leq \cdots \leq \hat{\phi}_{\sigma(N)}$. Since $\phi$ is monotone and $\hat{\phi}_N$  is close to $\phi$ or $-\phi$ (by Theorem~\ref{thm:FiedlerConv}), we expect the permutation $\sigma$ to be close to the identity permutation (or its reverse).

For any permutation $\pi \in \mathcal{S}_N$, define the Kendall tau distance to be
\begin{equation*}
    D_N(\pi) = \sum_{i < j} \mathbf{1}_{\{\pi(i) > \pi(j)\}}.
\end{equation*}
In \cite{diaconis1977spearman}, it was shown that, for any $\pi \in \mathcal{S}_N$, we have $D_N(\pi) \leq \| \pi - id_N\|_1 \leq 2D_N(\pi)$. Thus, when $N \to \infty$, these distances are of the same order. 

We are now ready to prove Theorem~\ref{ThmMainThm}:
\begin{proof} [Proof of Theorem \ref{ThmMainThm}]
Fix any $0 < \eta <\frac{1-\tau}{3}$. We assume without loss of generality that $\norm{\hat{\phi}_N-\phi} \leq  \norm{\hat{\phi}_N + \phi}$ (if not, simply replace $\phi$ with $-\phi$ in the following). Define
\begin{align*}
    E_N &:= \left\{(i,j) \in [N]^{2} \, : \, i<j \text{ and } \sigma(j)<\sigma(i)\right\}\\
    E_N^{\eta} &:= \left\{(i,j) \in [N]^{2} \, : \, i+N^{1-\eta}<j \text{ and } \sigma(j)<\sigma(i)\right\},\\
\end{align*}
We then have 
\begin{align*}
    2\norm{\hat{\phi}_N-\phi}^2 &= \int_0^1(\hat{\phi}_N(x)-\phi(x))^2dx + \int_0^1 (\hat{\phi}_N(y)-\phi(y))^2dy\\
    &= \sum_{i=1}^{N}\int_{\iI_i^N}(\hat{\phi}_i-\phi(x))^2dx+\sum_{j=1}^{N}\int_{\iI_j^N}(\hat{\phi}_j-\phi(x))^2dx\\
    &= \frac{1}{N}\sum_{1\leq i,j \leq N}\left(\int_{\iI_i^N}(\hat{\phi}_i-\phi(x))^2dx + \int_{\iI_i^N}(\hat{\phi}_j-\phi(x+\frac{j-i}{N}))^2dx\right)\\
    &\geq \frac{1}{N}\sum_{(i,j)\in E_N^{\eta}} \int_{\iI_i^N}\left((\hat{\phi}_i-\phi(x))^2 + (\hat{\phi}_j-\phi(x+\frac{j-i}{N}))^2 \right)dx.
\end{align*}
Since $\phi(x+\frac{j-i}{N})>\phi(x)$ and $\hat{\phi}_j < \hat{\phi}_i$ for $(i,j) \in E_N$, we also have
\begin{equation*}
    (\hat{\phi}_i-\phi(x))^2 + (\hat{\phi}_j-\phi(x+\frac{j-i}{N}))^2 
    \geq
    \frac{1}{2}\left(\phi(x+\frac{j-i}{N})-\phi(x)\right)^{2}.
\end{equation*}
Combining the previous two estimates,
\begin{align*}
    2\norm{\hat{\phi}_N-\phi}^2 &\geq \frac{1}{2N} \sum_{(i,j)\in E_N^{\eta}} \int_{\iI_i^N} \left(\phi(x+\frac{j-i}{N})-\phi(x)\right)^2 dx\\
    &\stackrel{\text{Theorem } \ref{thm:monotone}}{\geq} \frac{L^2}{2N^3}  \sum_{(i,j)\in E_N^{\eta}} \int_{\iI_i^N} (j-i)^{2} dx\\
    &\geq \frac{L^2}{2N^{1+2\eta}} \sum_{(i,j)\in E_N^{\eta}} \int_{\iI_i^N} dx\\
    &= \frac{L^2}{2N^{2+2\eta}} \Card{E_N^{\eta}},
\end{align*}
where the inequality from  line two to three follows from the fact that $j>i+N^{1-\eta}$ for $(i,j)  \in E_N^{\eta}$. Consequently,
\begin{equation*}
    \Card{E_N^{\eta}} \leq \frac{4}{L^2}N^{2+2\eta}\norm{\hat{\phi}_N-\phi}^2.
\end{equation*}
On the other hand, 
\begin{equation*}
E_{N} \setminus E_N^{\eta}=\left\{(i,j) \in [N]^{2} : 0< j-i \leq  N^{1-\eta} \text{ and }  \sigma(j)<\sigma(i)\right\},
\end{equation*}
thus $\Card{E_{N}\setminus E_{N}^{\eta}}\leq N^{2-\eta}$. By Theorem~\ref{thm:FiedlerConv} there exists a positive constant $C = C(w)$ so that for all $N \geq N_0(w, \tau, \eta, \varepsilon)$ we have with probability at least $1-\varepsilon$
\begin{equation*}
     \norm{\hat{\phi}_N-\phi}^2 \leq C^2 N^{-3\eta}.
\end{equation*}
Recall that $D_N(\sigma) = \Card{E_N}$ and $\|\sigma - id_N\|_1 \leq \tfrac{1}{2} D_N(\sigma)$, therefore with probability at least $1 - \varepsilon$
\begin{align*}
    \| \sigma - id_N \|_1
    &\leq \frac{1}{2} \left( \Card{E_{N}\setminus E_{N}^{\eta}} + \Card{E_{N}^{\eta}} \right) \\
    &\leq  \frac{1}{2} N^{2-\eta} +  \frac{2}{L^2}N^{2+2\eta}\norm{\hat{\phi}_N-\phi}^2 \\
    &\leq (\frac{1}{2} + \frac{2C^2}{L^2})N^{2-\eta}
\end{align*}
as required. The same inequality holds for $\varepsilon = \varepsilon_{\gamma}(N)$ for all $N \geq N'_0(w, \tau, \eta, \gamma)$.
\end{proof}
\section{Postprocessing and Tighter Error Bounds} \label{SecPostProc}

The main goal of this section is to prove Theorem~\ref{ThmSomePostprocMain}. We introduce a post-processing algorithm of the estimate returned by Algorithm~\ref{Alg:SpectralSeriation}. This new algorithm provides an estimate that is more robust to noise, with a nearly optimal convergence rate.

\subsection{Notation}

We set notation that will be used throughout Section \ref{SecPostProc}. Let $G = (V, E)$ be a graph of size $N$. Denote by $V := \left\{v_1, v_2, \ldots, v_N\right\}$ the vertices of $G$, with $v_i = i/N$. For any $S \subset V$, denote the \textit{induced subgraph} on $S$ by $G|_{S}$. Define the (restricted) \textit{neighbourhood} of a vertex $v \in V$ by $N_{S}(v) = \{ u \in S \, : \, (u, v) \in E\}$. Denote by $\sigma_{S} \, : \, S \to [|S|]$ the function that sends $s \in S$ to $\Card{\{ s' \in S \, : \, s' \leq s \}}$; we think of this ordering as being essentially the ``identity'' permutation on $S$.

Algorithm \ref{AlgPostProc} involves partitioning up the vertices of $G$, estimating various quantities on parts of the partition, and then combining them. We say that a partition $(S_1, S_2, S_3)$ of $V$ is ``\textit{good}'' if $n \leq |S_3| \leq |S_2| \leq |S_1| \leq n+1$, where $n = \left \lfloor \frac{N}{3} \right \rfloor$. 

Let $(S_1, S_2, S_3)$ be a ``good'' partition of $V$ and let $\sigma_j : S_j \to [|S_j|]$ be orderings of each part of the partition. We define the ``\textit{merged ordering}'' of $(\sigma_1, \sigma_2, \sigma_3)$ as an ordering $\sigma : [N] \to [N]$ defined by: 
\begin{equation}\label{eq:merged-permutation}
    \sigma^{-1}(3k+j) = N \sigma_{j}^{-1}(k+1),
\end{equation}
for all  $k \in \{0, 1, \ldots, n\}$ and $j \in \{1,2,3\}$ such that $3k+j \leq N$. Note that Equation \eqref{eq:merged-permutation} does define a permutation, since for all $i \in [N]$ there is exactly one solution to the equation $i = 3k+j$ with $j \in \{1,2,3\}$ and $k \in \{0,1,\ldots,n\}$. Note that the factor $N$ in the definition is coming from the fact that the elements of $V$ are scaled by $v_i = i/N$.

Finally, for two orderings $\sigma_{1}, \sigma_{2}$ on sets $S_{1},S_{2}$, say that $\sigma_{1},\sigma_{2}$ \textit{are aligned} if 
\begin{equation} \label{EqAlignDef}
    \max_{i \in S{1} \cap S_{2}} \left|\frac{\sigma_{1}(i)}{|S_{1}|} - \frac{\sigma_{2}(i)}{|S_{2}|}\right| \leq   \max_{i \in S_{1} \cap S_{2}} \left|\frac{\mathrm{rev}(\sigma_{1})(i)}{|S_{1}|} - \frac{\sigma_{2}(i)}{|S_{2}|}\right|,
\end{equation}

and that $\sigma_1, \sigma_2$ \textit{not aligned} otherwise. If $S_{1} \cap S_{2} = \emptyset$, say that they are aligned. Moreover, if the left-hand side of \eqref{EqAlignDef} is less than 0.001 and the right-hand side is more than 0.999, say that $\sigma_{1},\sigma_{2}$ are \textit{closely aligned}. Recall that $ \mathrm{rev}(\sigma)$ denotes the reversal of any ordering $\sigma$ defined in Equation~\eqref{eq:ReversalOrdering}. \\

%\newpage

\subsection{Algorithms and Informal Analysis}

We are now ready to present our main algorithms.

\begin{algorithm}[h] \label{AlgPostProcSplit}
    \SetKwInOut{Input}{Input}
    \SetKwInOut{Output}{Output}

    \Input{ A graph $G = (V, E)$ of size $N$, a pair of disjoint sets $(T, S) \subset V^2$ and parameters $0 < \alpha < \beta < 1/2$.}
    
     Let $\hat{\sigma}_T$ be the result of running Algorithm \ref{Alg:SpectralSeriation} on $G|_{T}$. \\
    Given $\hat{\sigma}_T$, estimate the set of right-most and left-most vertices in $T$ by
     \begin{align}
     R = R_{\hat{\sigma}_T}(\alpha) &:= \left\{u \in T \, : \, \hat{\sigma}_{T} (u) \geq (1-\alpha) \, |T| \right\}, \label{eq:alpha-right-most}\\
     L = L_{\hat{\sigma}_T}(\alpha) &:= \left\{u \in T \, : \, \hat{\sigma}_{T} (u) \leq \alpha \, |T| \right\}. \label{eq:alpha-left-most}
 \end{align} \\
    For each $v \in S$, estimate the neighbourhood-size functions
    \begin{align}
    \psi_R(v) &:= \frac{1}{|T|}\,\Card{N_{R}(v)} = \frac{1}{|T|} \,\sum_{u \in R} \mathbf{1}_{(u, v) \in E},  \label{eq:psiR-estimate} \\ 
    \psi_L(v) &:= \frac{1}{|T|} \, \Card{N_{L}(v)} = \frac{1}{|T|} \,\sum_{u \in L} \mathbf{1}_{(u, v) \in E}. \label{eq:psiL-estimate}
\end{align}
 \\
    Compute the following empirical quantiles of $\psi_{L}, \psi_{R}$:
    \begin{align} 
    c_R &:= \inf \left\{ c \geq 0 \, : \, \Card{\left\{v \in S \, : \, \psi_{R}(v) \leq c \right\}} \geq (1-\beta) |S| \right\}, \label{eq:CbetaR} \\
    c_L &:= \inf \left\{ c \geq 0 \, : \, \Card{\left\{v \in S \, : \, \psi_{L}(v) \leq c \right\}} \geq (1-\beta) |S| \right\}. \label{eq:CbetaL}
\end{align} \\

    Compute the antisymmetric function $\hat{F} \, : \, S^{2} \mapsto \{-1, 0, 1\}$ by the following rule for $u < v$:  
    \begin{equation}  \label{EqDefAntisymmetric1}
\hat{F}(u,v) =
\left\{
	\begin{array}{ll}
		1 - 2 \cdot \mathbf{1}_{\psi_{R}(u) > \psi_{R}(v)}  & \mbox{if } \psi_{R}(u), \psi_{R}(v) < c_{R}  \mbox{; otherwise,} \\ % The ``good" case that both are on the left half of the interval.
		1 - 2 \cdot \mathbf{1}_{\psi_{L}(u) < \psi_{L}(v)}  & \mbox{if } \psi_{L}(u), \psi_{L}(v) < c_{L}  \mbox{; otherwise,} \\ % The ``good" case that both are on the right half of the interval.
		 1  & \mbox{if } \mbox{$\psi_{R}(u) < c_{R}$} \mbox{; otherwise,} \\ % In this case $i$ is on the left, but $j$ isn't.
		 -1 &\mbox{} \\ % In this case $j$ is on the left, but $i$ isn't.
	\end{array}
\right\}.
    \end{equation}
 \\
 
 For each $v \in S$ compute
    \begin{equation}\label{eq:f-hat}
    \hat{f}(v) = \sum_{u \in S} \hat{F}(u,v).
    \end{equation}
 \\
 Compute the ordering $\hat{\sigma}_S = \hat{\sigma}_{\hat{f}}$ on $S$ according to Equation (\ref{Eq:DefFuncToPerm}), breaking ties arbitrarily. \footnote{Note that the range of $\hat{\sigma}_S$ is $\{1,2,\ldots,|S|\},$ which is generally not equal to the set $S$ itself.}\\

    \Output{ The permutation $\hat{\sigma}_S$.}
    \caption{Initial Post-processing the output of spectral seriation.}
\end{algorithm}

We give an informal summary of what this algorithm is doing and why we expect it to work well, eliding many details.

Steps 1-2 get a good estimate of the right-most and left-most vertices in $T$, denoted $R$ and $L$ respectively. Using these estimated sets, steps 3-6 try to compare vertices $u,v \in S$ by comparing the number of neighbours these vertices have in $R$ and $L$. The Robinsonian property suggests the following heuristic: if $u < v$, then $u$ should have more neighbours in $R$ while $v$ should have more neighbours in $L$. If this heuristic were exactly correct, we would use only the first of four cases in the function $\hat{F}$ from Equation \eqref{EqDefAntisymmetric1}. However, the heuristic isn't quite correct: it can fail if (i) both of $u,v$ are very close to either $0$ or $1$ (in which case many of the ``right-most'' elements of $R$ are not actually to the right of $u,v$ or many of the ``left-most'' elements of $L$ are not actually to the left of $u,v$) or (ii) if  $|u-v|$ is too small (in which case sampling error may dominate). 

The four cases in Equation \eqref{EqDefAntisymmetric1} deal with the first of these problems by ensuring that we compare $u,v$ by counting either right- or left-most neighbours as appropriate. To give a back-of-the-envelope analysis of the sampling error that can cause the second problem, we note that there are $\Theta(N)$ possible edges from $u,v$ to $R$ or $L$, and the presence or absence of these edges are independent of the random variables used to estimate $R$ and $L$ (the purpose of partitioning $V$ is exactly to obtain this independence). This suggests that, as long as the error in the estimates of $R$ and $L$ are ``negligible'', the typical error in the estimates of $\psi_{R},\psi_{L}$ should be $\mathcal{O}(\sqrt{N})$. In particular, this is the step that lets us clean rough estimates of sets $R, L$ into much better estimates of the averages $\psi_{R},\psi_{L}$. Although this may not be obvious at first glance, the consistency result in Theorem \ref{ThmMainThm} is strong enough to guarantee that the error in the estimates of $R$, $L$ are ``negligible'' in the sense that this argument requires. This calculations suggests that $\hat{\sigma}_S$ should have a sup-norm error that is not too much larger than $\mathcal{O}(\sqrt{N})$, which is the case.

We now give our overall post-processing algorithm :

\begin{algorithm}[H] \label{AlgPostProc}
    \SetKwInOut{Input}{Input}
    \SetKwInOut{Output}{Output}

    \Input{ A graph $G = (V, E)$ of size $N$. Parameters $0 < \alpha < \beta < 1/2$.}
    Choose $(S_1, S_2, S_3)$ a ``good'' partition of $V$, uniformly at random; and independently choose another ``good'' partition $(S_1', S_2', S_3')$, uniformly at random. \\

    Use Algorithm \ref{AlgPostProcSplit} on pairs of subsets $(S_{3},S_{1})$, $(S_{1},S_{2})$ and $(S_{2},S_{3})$ with parameters $\alpha, \beta$ to obtain estimated orders $\hat{\sigma}_{1},\hat{\sigma}_{2},\hat{\sigma}_{3}$ on $S_{1},S_{2},S_{3}$ respectively. Also use Algorithm \ref{AlgPostProcSplit} on $(S_{1}',S_{2}')$ with parameters $\alpha, \beta$ to obtain estimated order $\hat{\sigma}_{\mathrm{ref}}$ on $S_{2}'$. \\
    For $j \in \{1,2,3\}$, replace $\hat{\sigma}_{j}$ by $\mathrm{rev}(\hat{\sigma}_{j})$ if it is not aligned with $\hat{\sigma}_{\mathrm{ref}}$ in the sense of Equation \eqref{EqAlignDef}.\\
    Compute $\hat{\sigma}$ the merged ordering of $(\hat{\sigma}_1, \hat{\sigma}_2, \hat{\sigma}_3)$ according to Equation \eqref{eq:merged-permutation}.\\
    \Output{ The permutation $\hat{\sigma}$ .}
    \caption{Post-processing the output of spectral seriation.}
\end{algorithm}

Algorithm \ref{AlgPostProc} largely consists of many calls to Algorithm \ref{AlgPostProcSplit}, and so most of the algorithm seems straightforward. The one possibly-mysterious part is step 3. To explain the need for this step, recall that it is impossible to distinguish between a vertex ordering $\sigma$ and its reverse $\mathrm{rev}(\sigma)$ by looking only at the edges of a graph. Up to this point in the paper, this has not posed a problem, as we always looked at an entire graph all at once. However, in Algorithm \ref{AlgPostProc}, we use Algorithm \ref{AlgPostProcSplit} to repeatedly order different subsets of the vertices. As a result, it is likely that some of the returned orderings $\hat{\sigma}_j : S_j \to [|S_j|]$ will be aligned with $\sigma_{S_j}$ (in the sense of Equation \eqref{EqAlignDef}) while others will be aligned with $\mathrm{rev}(\sigma_{S_j})$. Fortunately, this problem is not difficult to solve: each ordering will either be closely aligned with $\sigma_{S_j}$ or its reverse, and so a single reference ordering $\hat{\sigma}_{\mathrm{ref}}$ can be used to simultaneously align \textit{all} other orderings, as long as their supports have non-empty intersection.

\subsection{Analysis of Algorithm \ref{AlgPostProcSplit}} \label{SecPostHeartAnalysis}

In this subsection, we prove the main bounds on the objects appearing in Algorithm~\ref{AlgPostProc}. We start by presenting the main theorem that provides a bound for the error of the estimated ordering $\hat{\sigma}_S$ returned by Algorithm \ref{AlgPostProcSplit}. Then the rest of this subsection will be dedicated to proving this theorem. 

\begin{thm} \label{thm:error-bound-alg2}
Let $G = (V, E)$ be a Robinsonian graph of size $N$ sampled according to a graphon $w$ as in Theorem~\ref{ThmSomePostprocMain}. Fix $s,t$ satisfying $ n\leq s, t \leq n+1$ and let $S, T \subset V$ be chosen uniformly at random from all disjoint subsets of size $|S|=s$, $|T|=t$. Consider $\hat{\sigma}_{S} : S \to [|S|]$ to be the output of Algorithm~\ref{AlgPostProcSplit} with input $G$, the pair of random sets $(T, S)$ and parameters $0 < \alpha < \beta < 1/2$ as in Assumption~\ref{AssumptionPsiOK}. For all $\gamma > 1$ and $\epsilon > 0$, there exists a positive constant $\delta = \delta(w, \gamma, \epsilon) > 0$ and an integer $N_0 = N_0(w, \gamma, \epsilon)$ so that
\[
 \Prob{\norm{\hat{\sigma}_{S} - \sigma_{S}}_1 \leq \epsilon \sqrt{N\log^{\gamma}(N)}} \geq 1 - \varepsilon_{\gamma}^{\delta}(N)
\]
for all $N \geq N_0$.
\end{thm}

Throughout this subsection we use notation as in Theorem~\ref{thm:error-bound-alg2}, let $0 < \alpha < \beta < 1/2$ as in Assumption~\ref{AssumptionPsiOK}, and set $\mu = (\alpha + \beta)/2$.

We fix a constant $\gamma > 1$ for the remainder of the section, and say that a sequence of events $\{ \mathcal{E}_{N} \}_{N \in \mathbb{N}}$ holds \textit{with extreme probability} if there exist a constant $\delta > 0$ and an integer $N_0$, such that 
\begin{equation}\label{eq:w.e.p}
    \Prob{\mathcal{E}_{N}} \geq 1 - \varepsilon_{\gamma}^{\delta}(N)
\end{equation}
for all $N \geq N_0$. We abbreviate this \textit{w.e.p}. Note that for any $C > 0$, the intersection of $\mathcal{O}(N^{C})$ events that occur \textit{w.e.p} for some $\gamma > 1$ also occur \textit{w.e.p} with the same $\gamma$. We define the sequence $\Delta_N = \sqrt{\tfrac{\log^{\gamma}(N)}{N}}$ for $N \geq 1$; this is essentially the scale of the most important errors that occur in the algorithm.

The first step in the proof is checking that the discrete functions $\psi_{R}, \psi_{L}$ are good approximations of the continuous functions $\Psi_{R}, \Psi_{L}$ respectively. We recall some facts about these functions. For $v \in V$, recall 
\begin{equation*}
    \Psi_R(v) = \int_{1-\alpha}^{1} w(v,u) \, du.
\end{equation*}

Denote by $\mathcal{F}_{T}$ the $\sigma$-algebra generated by the choice of $T$ and the graph $\restr{G^{(N)}}{T}$. Since $\psi_{R}(v)$ and $\psi_{L}(v)$ are sums of Bernoulli random variables sampled according to Equation~\eqref{EqDefSamplingRule} then
\begin{equation*}
    \mathbb{E}[\psi_{R}(v)|\mathcal{F}_T] = \frac{1}{|T|} \,\sum_{u \in R} w(u,v), \text{ and } \, \mathbb{E}[\psi_{L}(v)|\mathcal{F}_T] = \frac{1}{|T|} \,\sum_{u \in L} w(u,v);
\end{equation*}
in particular $R,L$ are both $\mathcal{F}_{T}$-measurable, while the indicator function of an edge $(u,v)$ with $u \in L \cup R$, $v \in S$ is \textit{independent} of $\mathcal{F}_{T}$ and has mean $w(u,v)$. This implies that for any fixed $\epsilon > 0$ the following holds \textit{w.e.p} 

\begin{equation} \label{IneqHoeffdingLike}
\max_{v \in S} \left|\psi_{R}(v) - \frac{1}{|T|} \,\sum_{u \in R} w(u,v)\right| \leq \epsilon \, \Delta_{N}.
\end{equation}

 The following Lemma shows that the graph-dependent functions $\psi_R, \psi_L$ approximate the graphon-dependent functions $\Psi_R, \Psi_L$

 \begin{lemma}\label{lemma:psiR-PsiR-bound}
For any $\epsilon > 0$, the  inequalities
 \begin{equation*}
     \max_{v \in S} \left|\psi_R(v)  - \Psi_{R}(v) \right| \leq \epsilon \Delta_N, \text{ and } \max_{v \in S} \left|\psi_L(v)  - \Psi_{L}(v) \right| \leq \epsilon \Delta_N
 \end{equation*}
  hold \textit{w.e.p.}
\end{lemma} 
Similarly, the constants $c_R$ and $c_L$ are intended to estimate $\Psi(1-\beta)$ and $\Psi(\beta)$ respectively. 
\begin{lemma}\label{lemma:CR-CL-bound}
For any $\epsilon > 0$ the following inequalities 
 \begin{equation*}
     \left|c_R - \Psi_R(1-\beta)\right| \leq \epsilon\Delta_N, \, \text{ and } \, \left|c_L - \Psi_L(\beta)\right| \leq \epsilon\Delta_N
 \end{equation*}
hold \textit{w.e.p.}
 \end{lemma}
 
  \begin{proof}[Proof of Lemma~\ref{lemma:psiR-PsiR-bound}]
 We will present the proof only for $\psi_R$, since the proof for $\psi_L$ is identical. Denote by $T = \{t_{(1)},\ldots,t_{(t)}\}$ the elements of $T$ in order, so that $t_{(1)} < t_{(2)} < \ldots < t_{(t)}$. Note that $\sigma_T(t_{(k)}) = k$ for all $k$. For $v\in S$, we have
 \begin{equation*}
     \mathbb{E}[\psi_R(v) | \mathcal{F}_T] = \frac{1}{t}\sum_{k = \ceil{(1-\beta)t}}^{t} w(t_{(k)}, v).
 \end{equation*}
Now define 
\[
\Sigma := \frac{1}{N}\sum_{k = \ceil{(1-\beta)N}}^{N} w(v_k, v),
\]
and, for each $j = 0, 1, 2$, define
\[
\Sigma_j := \frac{1}{t}\sum_{k = \ceil{(1-\beta)t}}^{t} w(v_{3k+j}, v).
\]
By the triangle inequality,
\begin{equation}\label{eq:triangular-inequality}
    \left| \mathbb{E}[\psi_R(v) | \mathcal{F}_T] - \Psi_{R}(v) \right |
    \leq 
    \left|\Psi_{R}(v) - \Sigma \right|
    +
    \frac{1}{3}\sum_{j=0}^{2}\left| \mathbb{E}[\psi_R(v) | \mathcal{F}_T] - \Sigma_j\right | 
    + 
    \frac{1}{3}\left| \Sigma_0 + \Sigma_1 + \Sigma_2 - 3\Sigma \right |.
\end{equation}
We now bound each term of the right hand side in inequality~\eqref{eq:triangular-inequality}. Recall the constant $K$ from Assumption~\ref{assumption:lipschitz}. We have
\begin{equation}\label{ineq:result1}
    \left|\Psi_{R}(v) - \Sigma \right| \leq \frac{\beta K}{N},
\end{equation}
by essentially the same argument as the proof of the first inequality in the proof of Lemma~\ref{lemma:MaxDegree}. Next,
\begin{align*}
    |\Sigma_0 + \Sigma_1 + \Sigma_2 - 3\Sigma| 
    & \leq \left|\Sigma_0 + \Sigma_1 + \Sigma_2 -\frac{N}{t}\Sigma \right| + \frac{|N-3t|}{t} \Sigma \\
    & = \frac{1}{t} \underset{\leq (1-\beta)|N-3t|}{\underbrace{\left| \sum_{k = \ceil{3(1-\beta)t}}^{N} w(v_k, v) - \sum_{k = \ceil{(1-\beta)N}}^{N} w(v_k, v) \right|}} + \underset{\leq \beta}{\underbrace{\Sigma}} \, \frac{|N-3t|}{t} \\ 
    & \leq \frac{|N-3t|}{t},
\end{align*}
where the bound in the first underbrace comes from the facts that (i) at most $(1-\beta)|N-3t|$ terms appear in one sum but not the other and (ii) each summands $w(v_{k},v)$ is bounded by 1, and the bound in the second underbrace comes again from the fact that the summands of $\Sigma$ are bounded by 1. Since $3t-3 \leq N \leq 3t+2$ by assumption, we conclude
\begin{equation}\label{ineq:result2}
     |\Sigma_0 + \Sigma_1 + \Sigma_2 - 3\Sigma| \leq \frac{12}{N}.
\end{equation}

Finally, fixing $ 0 \leq j \leq 2$, we calculate
\begin{align*}
    \left| \mathbb{E}[\psi_R(v) | \mathcal{F}_T] - \Sigma_j \right |
    & \leq \frac{1}{t}\sum_{k \geq (1-\beta)t}\left| w(t_{(k)}, v) - w(v_{3k+j}, v) \right|\\
    & \leq  \frac{K}{t} \sum_{k \geq (1-\beta)t}|t_{(k)} - v_{3k+j}|.
\end{align*}
According to Lemma~\ref{lemma:order-stat-bound},
\begin{equation*}
   \max_{k \in T} |t_{(k)}- \frac{k}{t}| \leq \frac{\epsilon}{2\beta K}\Delta_N \hspace{0.5cm} \textit{w.e.p.}
\end{equation*}
Additionally, it is easy to check that $|\frac{k}{t} - v_{3k+j}| \leq 8/N $. Therefore we have with extreme probability
\begin{equation*}
    |t_{(k)} - v_{3k+j}| \leq \frac{\epsilon}{2\beta K}\Delta_N + \frac{8}{N}. 
\end{equation*}
We conclude
\begin{equation}\label{ineq:result3}
    \left| \mathbb{E}[\psi_R(v) | \mathcal{F}_T] - \Sigma_j \right | \leq   \frac{\epsilon}{2}\Delta_N  + \frac{8\beta K}{N} \hspace{0.5cm} \textit{w.e.p.}
\end{equation}
We substitute Inequalities \eqref{ineq:result1}, \eqref{ineq:result2} and \eqref{ineq:result3} in Inequality \eqref{eq:triangular-inequality} to get
\begin{equation*}
    \left| \mathbb{E}[\psi_R(v) | \mathcal{F}_T] - \Psi_{R}(v) \right | \leq \frac{\epsilon}{2}\Delta_N + \frac{9\beta K + 12}{N} \leq \epsilon \Delta_N \hspace{0.5cm} \textit{w.e.p.}
\end{equation*}
Applying the concentration bound \eqref{IneqHoeffdingLike} completes the proof.
 \end{proof}
 
 Define 
 \begin{equation} \label{EqDefM}
 M = \sup\limits_{x \in [0, 1]} \Psi_{R}'(x) > 0;
 \end{equation}
 this exists by Assumption ~\ref{assumption:PartialDerivativeGraphon}.
 
\begin{proof}[Proof of Lemma~\ref{lemma:CR-CL-bound}]
We give the proof only for $\Psi_{R}$, as the case $\Psi_{L}$ is identical. Recall that by definition of $c_R$
\begin{equation*}
    \Card{ \left\{ v \in S \, : \, \psi_R(v) \leq c  \right\} } \geq (1-\beta) s \Leftrightarrow c \geq c_R
\end{equation*}
for all $c > 0$.  For $v \in V$, we have
\begin{equation*}
  \{  |v - (1 - \beta)| \leq  \frac{\epsilon\Delta_N}{2M} \}
     \Rightarrow   \{ |\Psi_R(v) - \Psi_R(1-\beta)| \leq \frac{\epsilon\Delta_N}{2} \}.
\end{equation*}
By Lemma~\ref{lemma:psiR-PsiR-bound}, we have 
\[
|\psi_R(v) - \Psi_R(v)| \leq \frac{\epsilon\Delta_N}{2} \hspace{0.5cm} \textit{w.e.p.}
\]
Combining the last two bounds,
\begin{align*}
     |v - (1 - \beta)| \leq  \frac{\epsilon \Delta_N}{2M}
     & \Rightarrow |\psi_R(v) - \Psi_R(1-\beta)| \leq  \epsilon\Delta_N \hspace{0.5cm} \textit{w.e.p.}
\end{align*}
Let $\{s_{(k)}\}_{k=1}^{s}$  be the elements of $S$ in increasing order and let $p = \ceil{(1-\beta)s}$. By Lemma~\ref{lemma:order-stat-bound}, \begin{equation*}
    |s_{(p)}-(1-\beta)| \leq |s_{(p)}-\frac{p}{s}| + |\frac{p}{s}-(1-\beta)| \leq \frac{\epsilon \Delta_N}{2M} \hspace{0.5cm} \textit{w.e.p.}
\end{equation*}
By part 1 of Assumption \ref{AssumptionPsiOK},
    \[
    \psi_R(s_{(k)}) \leq \Psi_R(1-\beta) + \epsilon\Delta_N, \hspace{0.5cm} \textit{w.e.p},
    \]
    for all $1 \leq k \leq p$ and 
    \[
    \Psi_R(1-\beta) - \epsilon\Delta_N\leq \psi_R(s_{(k)}) \hspace{0.5cm} \textit{w.e.p},
    \]
    for all $p \leq k \leq s$. Therefore
    \[
    \Card{ \left\{ v \in S \, : \, \psi_R(v) \leq \Psi_R(1-\beta) + \epsilon\Delta_N  \right\} } \geq p \geq (1-\beta) s, \hspace{0.5cm} \textit{w.e.p},
    \]
     which implies that 
     \begin{equation}\label{IneqCrUpper}
     c_R \leq \Psi_R(1-\beta) + \epsilon\Delta_N \hspace{0.5cm} \textit{w.e.p.}
     \end{equation}
     On the other hand,
    \[
    \Card{ \left\{ v \in S \, : \, \psi_R(v) \geq \Psi_R(1-\beta) - \epsilon\Delta_N  \right\} } \geq s - p + 1 > \beta s \hspace{0.5cm} \textit{w.e.p},
    \]
    this is equivalent to
    \[
    \Card{ \left\{ v \in S \, : \, \psi_R(v) \leq \Psi_R(1-\beta) - \epsilon\Delta_N  \right\} } < (1-\beta)s \hspace{0.5cm} \textit{w.e.p},
    \]
    which implies that 
    \begin{equation}\label{IneqCrLower}
    c_R > \Psi_R(1-\beta) - \epsilon\Delta_N \hspace{0.5cm} \textit{w.e.p.}
    \end{equation}
    Combining Inequalities \eqref{IneqCrUpper} and \eqref{IneqCrLower}, we deduce that 
    \[
    |c_R - \Psi(1-\beta)| \leq \epsilon \Delta_N \hspace{0.5cm} \textit{w.e.p.}
    \]
    The proof for $c_L$ is identical to the proof of $c_R$.
\end{proof}

The following proposition states that if we have two vertices $u, v \in S$ such that (i) $u$ and $v$ are not \textit{very close} to each other, and (ii) both of them belong to one of the intervals $[0, 1-\mu]$ or $[\mu, 1]$. Then these vertices can be ordered accurately using the graph-dependent functions $\psi_{R}$ and $\psi_{L}$ \textit{w.e.p}. 
\begin{prop}\label{prop:accurate-order}
Let $\epsilon > 0$. With extreme probability,  the events
\begin{align*}
    &  \{ 0 \leq u < v \leq 1-\mu, \, \text{ and } \, v > u + \epsilon \Delta_N \} \Rightarrow \{ \psi_{R}(v) \geq \psi_{R}(u) \}, \\ 
    & \{ \mu \leq u < v \leq 1, \,  \text{ and } \, v > u + \epsilon \Delta_N \} \Rightarrow \{ \psi_{L}(v) \leq \psi_{L}(u) \}
\end{align*}
hold for all pairs $(u, v) \in S^2$.
\end{prop}

\begin{proof}[Proof of Proposition~\ref{prop:accurate-order}]
Recall that $\hat{\sigma}_{T}$ from Algorithm~\ref{AlgPostProcSplit} is the ordering obtained from Algorithm~\ref{Alg:SpectralSeriation} with input $\restr{G}{T}$. By Theorem \ref{ThmMainThm}, we know that for any $\eta \in (0, 1/3)$ there exists a constant $C > 0$ such that 
\begin{equation}\label{IneqPostHeur1}
    \norm{\hat{\sigma}_{T} - \sigma_{T}}_{1} \leq  C \, t^{2 - \eta}
\end{equation}
holds \textit{w.e.p}. Note that the difference between $\sigma_{T}$ and its reverse is on the order of $t^2$. This means that, when Inequality \eqref{IneqPostHeur1} holds and $N$ is sufficiently large (equivalently $t$ is sufficiently large), $\hat{\sigma}_T$ will be \textit{much} closer to one of $\sigma_{T}$ or its reverse  (nothing close to a tie can occur). To simplify notation, in what follows, we assume without loss of generality that $\hat{\sigma}_T$ is closer to $\sigma_{T}$.

Fix $(u,v)\in S^2$ such that $0 \leq u < v \leq 1-\mu$ and $v > u + \epsilon \Delta_N$. Choose any $\kappa_1, \kappa_2$ such that $\alpha < \kappa_1 < \kappa_2 < \mu$, recall the definition of $R_{\sigma}(\kappa_1)$ from Equation~\eqref{eq:alpha-right-most},
 \begin{align*}
 \mathbb{E}[\psi_{R}(v) - \psi_{R}(u) | \mathcal{F}_{T}] 
 & = \frac{1}{t}\sum_{r \in R} (w(v,r) - w(u,r))  \\
 & = \frac{1}{t}\sum_{r \in R_{\sigma_T}(\kappa_1)} (w(v,r) - w(u,r)) + \frac{1}{t}\sum_{r \in R \setminus R_{\sigma_T}(\kappa_1)} (w(v,r) - w(u,r)).
 \end{align*}
 In what follows we will provide a lower bound for the the two summations appearing in $\mathbb{E}[\psi_{R}(v) - \psi_{R}(u) | \mathcal{F}_{T}]$. In fact there are two cases:
 \begin{enumerate}
     \item $r \in R \setminus R_{\sigma_T}(\kappa_1)$ : Let $K > 0$ be the constant from Assumption \ref{assumption:lipschitz}.
     \begin{equation} \label{IneqDeltaWPart1}
     w(v, r) - w(u,r) \geq -K(v-u)
     \end{equation}
     \item $r \in R_{\sigma_T}(\kappa_1)$ : In this case, $r = t_{(k)}$ for some $k > (1-\kappa_1)t$, where $t_{(k)}$'s are the order statistics of the elements in $T$. By Lemma~\ref{lemma:order-stat-bound},
     \begin{equation*}
        r = t_{(k)} \geq \frac{k}{t} - \Delta_N > (1-\kappa_1) - \Delta_N > (1-\kappa_2) \hspace{0.5cm} \textit{w.e.p.}
    \end{equation*}
    That is
    \begin{equation*}
        r \in R_{\sigma_T}(\kappa_1) \Rightarrow 0 \leq u < v \leq 1-\mu < 1-\kappa_2 < r \leq 1 \hspace{0.5cm} \textit{w.e.p.}
    \end{equation*}
    Therefore $\min(|v-r|, |u-r|) \geq \mu - \kappa_2 > \frac{\beta-\alpha}{2}$ \textit{w.e.p.}  Applying Inequality~\eqref{ineq:weaker-assumption} from Assumption~\ref{AssumptionPsiOK}, there exists some $d_1 >0$ such that
    \begin{equation} \label{IneqDeltaWPart2} 
         w(v, r) - w(u,r) \geq d_1(v-u) \hspace{0.5cm} \textit{w.e.p.}
    \end{equation}
 \end{enumerate}
Combining Inequalities \eqref{IneqDeltaWPart1} and \eqref{IneqDeltaWPart2},
\begin{equation*}
    \mathbb{E}[\psi_{R}(v) - \psi_{R}(u) | \mathcal{F}_{T}] \geq \frac{d_1}{t} \Card{R_{\sigma_T}(\kappa_1)}(v-u) - \frac{K}{t} \Card{R \setminus R_{\sigma_T}(\kappa_1)}(v-u) \hspace{0.5cm} \textit{w.e.p.}
\end{equation*}
Clearly $\Card{R_{\sigma_T}(\kappa_1)} \geq \kappa_1 t$. Moreover, if $r \in R \setminus R_{\sigma_T}(\kappa_1)$ then   $\hat{\sigma}_T(r) -\sigma_T(r) \geq (\kappa_1 - \alpha) \, t > 0$, thus
\begin{equation}\label{eq:inequality-sigmaHatT-sigmaT}
    \norm{\hat{\sigma}_T -\sigma_T}_1 \geq (\kappa_1 - \alpha) \, t \, \Card{R \setminus R_{\sigma_T}(\kappa_1)}.
\end{equation}
Finally we have \textit{w.e.p.},
\begin{align*}
    \mathbb{E}[\psi_{R}(v) - \psi_{R}(u) | \mathcal{F}_{T}]
    &\geq  \frac{d_1}{t} \Card{R_{\sigma_T}(\kappa_1)}(v-u) - \frac{K}{t} \Card{R \setminus R_{\sigma_T}(\kappa_1)}(v-u) \\
    &\stackrel{\text{Ineq } \eqref{eq:inequality-sigmaHatT-sigmaT}}{\geq}  \kappa_1 d_1  (v-u) - \frac{K}{(\kappa_1 - \alpha) t^2}  \norm{\hat{\sigma}_T -\sigma_T}_1 (v-u) \\
    &\stackrel{\text{Ineq } \eqref{IneqPostHeur1}}{\geq}  \left(\kappa_1 d_1   - \frac{C K}{\kappa_1 - \alpha} t^{-\eta}   \right) (v-u) \\ 
    &\geq \frac{\kappa_1 d_1}{2}(v-u) \\ 
    &\geq \frac{\epsilon \kappa_1 d_1}{2} \Delta_N.
\end{align*}
Now, since $\psi_R(v) - \psi_R(u)$ is a weighted sum of Bernoulli random variables sampled without uniformly without replacement. We can use Hoeffding's inequality for sampling without replacement \cite{serfling1974probability, greene2017exponential} to see that the conditional probability bound
 \begin{align*} 
     \mathbb{P}[\psi_{R}(v) \geq \psi_{R}(u) | \mathcal{F}_{T}]
     & \geq 1- \exp\left(-\frac{t^2}{2N} \left(\frac{\epsilon \kappa_1 d_1}{2}\right)^2 \Delta^2_{\gamma}(N) \right) \\
     &\geq 1- \exp(- \delta N\Delta^2_{\gamma}(N) ) \\
     &= 1 - \varepsilon_{\gamma}^{\delta}(N)
 \end{align*}
 holds \textit{w.e.p}\footnote{When we say that an inequality holds \textit{w.e.p} we must specify what is being viewed as random. To readers unfamiliar with random measures, this  might not be obvious in the preceding string of inequalities. We emphasize here that we are viewing the probability $\mathbb{P}[\psi_{R}(j) < \psi_{R}(i) | \mathcal{F}_{T}]$ itself as an $\mathcal{F}_{T}$-measurable random variable (and we are viewing $\mathbb{P}[\cdot | \mathcal{F}_{T}]$ as a random measure). Since $\mathcal{F}_{T}$ is generated by the random choice of $T$ and the many independent Bernoulli random variables contained in $A|_{T}$, it is $T$ and $A|_{T}$ that are viewed as random in this expression.} for some $\delta > 0$. The proof for $\psi_L$ is identical to the proof of $\psi_R$.
\end{proof}

We now list some further facts about $\psi_{R},\psi_{L}$:
\begin{prop} \label{PropLotsOfPsiInfo}
Fix $\epsilon > 0$. The following events hold \textit{w.e.p}:
\begin{enumerate}
    \item For all $v \in S $,
    \begin{align*}
    & \{\psi_{R}(v) < c_{R} \}\Rightarrow \{ v \leq 1-\mu\}, \\
    & \{ \psi_{L}(v) < c_{L} \}\Rightarrow \{ v \geq \mu \}.
    \end{align*}
    \item For all pairs $(u,v) \in S^2$,
\begin{align*}
     \{\psi_R(u) < c_R , \, \psi_{R}(v) > c_{R} \} &\Rightarrow \{ u < v \, \text{ or } \, |u-v| \leq \epsilon \Delta_N \}, \\ 
     \{ \psi_L(u) < c_L, \, \psi_{L}(v) > c_{L} \} &\Rightarrow \{ u > v \, \text{ or } \, |u-v| \leq \epsilon \Delta_N \}.
\end{align*}
\end{enumerate}
 \end{prop}
 
 \begin{proof}[Proof of Proposition~\ref{PropLotsOfPsiInfo}]~
 
 \begin{enumerate}
     \item  According to Lemmas~\ref{lemma:psiR-PsiR-bound} and \ref{lemma:CR-CL-bound}, with extreme probability we have
     \[
     \Psi_{R}(v) \leq \Psi_{R}(1-\beta) + \Delta_N < \Psi_{R}(1-\mu) \hspace{0.5cm} 
     \]
     for all $v \in S$ such that $\psi_R(v) < c_R$. By part 1 of Assumption \ref{AssumptionPsiOK}, this implies that \textit{w.e.p.} we have $\max \{v \in S \, : \, \psi_{R}(v) < c_{R} \} < 1-\mu$. The proof for $\psi_L$ is very similar.
     \item Recall $M$ from Equation \eqref{EqDefM}, define $m = \inf_{x \in [0,1-\mu]} \Psi_{R}'(x)$. Consider  $u, v \in S$ such that $\psi_R(u) < c_R$ and $\psi_R(v) > c_R$. If $u \geq v$ then according to the previous proposition $0 \leq v \leq u \leq 1-\mu$. Moreover, applying Lemma~\ref{lemma:psiR-PsiR-bound} and \ref{lemma:CR-CL-bound} with $\epsilon' = \tfrac{\epsilon m}{2}$ we get
     \[
     \Psi_{R}(1-\beta) - \frac{\epsilon m}{2}\Delta_N \leq \Psi_{R}(v) \leq \Psi_{R}(u) \leq \Psi_{R}(1-\beta) + \frac{\epsilon m}{2}\Delta_N \hspace{0.5cm} \textit{w.e.p.}
     \]
     Therefore
     \begin{equation*}
      |u-v| 
      \leq 
      \frac{1}{m} \left|\Psi_{R}(u) - \Psi_{R}(v) \right|
      \leq \Delta_N \hspace{0.5cm} \textit{w.e.p.}
     \end{equation*}
     The proof for $\psi_L$ is very similar.
 \end{enumerate}
 \end{proof}

Define the sets
 \begin{align*}
    \mathcal{E}_{1} &= \{ (u,v) \in S^2 \, : \, \psi_{R}(u), \psi_{R}(v) < c_R \}, \\
    \mathcal{E}_{2} &= \{ (u,v) \in S^2 \, : \, \psi_{L}(u), \psi_{L}(v) < c_L \}, \\
    \mathcal{E}_{3} &= \{ (u,v) \in S^2 \, : \, \psi_{R}(u) < c_R \}.
 \end{align*}
Recall the antisymmetric function $\hat{F} : S^2 \to \{-1, 0, 1\}$ defined in Equation~\eqref{EqDefAntisymmetric1} by
\begin{equation}\label{eq:DefAntisymmetric1}
\hat{F}(u,v) =
\left\{
	\begin{array}{ll}
		1 - 2 \cdot \mathbf{1}_{\psi_{R}(u) > \psi_{R}(v)}  & \mbox{if } (u,v) \in  \mathcal{E}_{1},  \\ 
		1 - 2 \cdot \mathbf{1}_{\psi_{L}(u) < \psi_{L}(v)}  & \mbox{if } (u,v) \in  \mathcal{E}_{2}, \cap \mathcal{E}_{1}^{c}, \\
		 1  & \mbox{if } (u,v) \in  \mathcal{E}_{3} \cap \mathcal{E}_{1}^{c} \cap \mathcal{E}_{2}^{c},\\ 
		 -1 &\mbox{if } (u,v) \in  \mathcal{E}_{1}^{c} \cap \mathcal{E}_{2}^{c} \cap \mathcal{E}_{3}^{c}.\\ 
	\end{array}
\right.
\end{equation}
 Define the antisymmetric function $F: S^2 \to \{-1, 0, 1\}$ by $F(u, v) = 1 - 2 \cdot \mathbf{1}_{u>v}$. Note that the function $\hat{F}$ appearing in Equation~\eqref{eq:DefAntisymmetric1} is intended to estimate the function $F$. Informally, Proposition~\ref{prop:accurate-order} tells us that the first two branches of $\hat{F}$ appearing in Equation~\eqref{eq:DefAntisymmetric1} would agree with $F$ for all pairs of vertices $u, v \in S \cap [0, 1-\mu]$ if $u$ and $v$ are not \textit{very close} and all pairs $u, v \in S \cap [\mu, 1]$ if $u$ and $v$ are not \textit{very close}. However, the function $\hat{F}$ defined in Equation \eqref{eq:DefAntisymmetric1} does not ``know'' if $u, v$ are in $[0, 1-\mu]$ or $[\mu, 1]$; it uses the four branches in its definition to guess. The following Lemma shows that the function $\hat{F}(u, v)$ defined in Equation~\eqref{eq:DefAntisymmetric1} is equal to the function $F(u, v)$ for all pairs of vertices $u, v$ that are not \textit{very close}.
 \begin{lemma}\label{lemma:almost-correct-order}
Let $\epsilon > 0$. The event
 \begin{equation}\label{FEqFTrue}
     \{v > u + \epsilon \Delta_N\} \Rightarrow \{ \hat{F}(u,v) = F(u, v)\}
 \end{equation}
for all pairs $(u, v) \in S^2$ holds \textit{w.e.p}.
 \end{lemma}

 \begin{proof}[Proof of Lemma~\ref{lemma:almost-correct-order}]
Let $(u,v)\in S^2$ such that $v > u + \epsilon \Delta_N$. We prove Equation \eqref{FEqFTrue} via a case-by-case analysis of the performance of $\hat{F}$ in its four branches. In fact we have four cases:
  \begin{enumerate}
     \item \textbf{$\mathcal{E}_{1}$ holds}: By the first part of Proposition \ref{PropLotsOfPsiInfo}, we have $u,v < 1-\mu$ \textit{w.e.p}. Applying Proposition~\ref{prop:accurate-order} completes the proof of Equation \eqref{FEqFTrue} in this case.
     \item \textbf{$\mathcal{E}_{2} \cap \mathcal{E}_{1}^{c}$ holds}: By the first part of Proposition \ref{PropLotsOfPsiInfo}, we have $u, v > \mu$ \textit{w.e.p}. Applying Proposition~\ref{prop:accurate-order} completes the proof in this case.
     \item \textbf{$\mathcal{E}_{3} \cap \mathcal{E}_{1}^{c} \cap \mathcal{E}_{2}^{c}$ holds:} By the second part of  Proposition \ref{PropLotsOfPsiInfo}, there are two possibilities. Either $\hat{F}$ makes the correct decision, or $|u-v|  \leq \epsilon\Delta_N$ \textit{w.e.p}.
     \item \textbf{$ \mathcal{E}_{1}^{c} \cap \mathcal{E}_{2}^{c} \cap \mathcal{E}_{3}^{c}$ holds:} Again by the second part of Proposition \ref{PropLotsOfPsiInfo}, there are two possibilities. Either $\hat{F}$ makes the correct decision, or $|u-v|  \leq \epsilon\Delta_N$ \textit{w.e.p}.
 \end{enumerate}

 \end{proof}
 
 Similarly, recall the function $\hat{f}$ defined in Equation~\eqref{eq:f-hat}. The following Lemma shows that the function $\hat{f}$ makes the correct ordering ``guess'' almost all the time; that is for $(u,v) \in S^2$ such that $u$ and $v$ are not \textit{very close} we have $(\hat{f}(v) - \hat{f}(u))(v-u) > 0$.
\begin{lemma}\label{lemma:almost-correct-ordering-function}
Fix $\epsilon > 0$. The event that 
 \begin{equation*}
     \{ v > u + \epsilon \Delta_N \} \Rightarrow \{ \hat{f}(v) > \hat{f}(u) \}
 \end{equation*}
holds for all pairs $(u, v) \in S^2$ occurs \textit{w.e.p}.
 \end{lemma}

 \begin{proof}[Proof Lemma~\ref{lemma:almost-correct-ordering-function}]
Fix $(u, v) \in S^2$ with $v >u + \epsilon\Delta_N$.  Define the set 
\[
\mathcal{A} = \left\{ r \in V \, : \,  \min(|u-r|, |v-r|) > \frac{\epsilon}{8}\Delta_N \right\}.
\] 
By  Lemma~\ref{lemma:almost-correct-order}, for all $r \in \mathcal{A}$
\begin{equation*}
    \hat{F}(r, v) = F(r, v) \, \text{ and } \, \hat{F}(r, u) = F(r, u) \hspace{0.5cm} \textit{w.e.p.}
\end{equation*}
Therefore we have \textit{w.e.p.}
\begin{align*}
    \hat{f}(v) - \hat{f}(u)
    &= \sum_{r \in V} \mathbf{1}_{r \in S} \left(\hat{F}(r, v) - \hat{F}(r, u)\right) \\
    &= \sum_{r \in \mathcal{A}} \mathbf{1}_{r \in S}\left(F(r, v) - F(r, u)\right)
    +
    \sum_{r \in\mathcal{A}^c} \mathbf{1}_{r \in S}\left(\hat{F}(r, v) - \hat{F}(r, u)\right) \\ 
    &\geq 2\Card{ \left\{ r \in S \, : \, u < r < v \right\} }
    -2\Card{S \cap \mathcal{A}^c} \\ 
    & = 2(A - B),
\end{align*}
where
\begin{equation*}
    A = \sum_{k \in E_1}  \mathbf{1}_{v_k \in S}, 
    \, \, 
    B = \sum_{k \in E_2}\mathbf{1}_{v_k \in S},
\end{equation*}
with $E_1 = \left\{ k \in [N] \, : \, Nu < k < Nv \right\}$ and $E_2 = \left\{ k \in [N] \, : \, \min(|k-Nu|, |k-Nv|) \leq \left \lfloor \frac{\epsilon}{8}N \Delta_N \right \rfloor \right\}$. Note that,  $A$ and $B$ appearing in this bound are sums of Bernoulli random variables sampled without replacement. That is $A, B$ are hypergeometric random variables, with $A \sim \mathcal{HG}(N; Nv-Nu-1; s)$ and $B \sim \mathcal{HG}(N;4\left \lfloor \frac{\epsilon}{8}N \Delta_N \right \rfloor + 2;s)$. We have
\begin{equation*}
    \mathbb{E}[A-B] = \frac{s}{N}\left(N(v-u)-4\left \lfloor \frac{\epsilon}{8}N \Delta_N \right \rfloor -3\right) > \frac{\epsilon}{10} N \Delta_{N} \hspace{0.5cm} \textit{w.e.p.}
\end{equation*}
Thus, by the usual analogue to Hoeffding's inequality for hypergeometric random variables \cite{serfling1974probability, greene2017exponential}, we have 

\begin{equation} \label{IneqGammaBound}
     \hat{f}(v) - \hat{f}(u) > 0 \hspace{0.5cm} \textit{w.e.p.}
\end{equation}
 \end{proof}

We are now ready to prove Theorem~\ref{thm:error-bound-alg2}.
\begin{proof}[Proof of Theorem~\ref{thm:error-bound-alg2}]
Let $s_{(k)}$'s be the order statistics of the set $S$. Fix $1 \leq k \leq s$, recall that 
\[
\hat{\sigma}_{S}(s_{(k)}) = \Card{\left\{ p \in [s] \, : \, \hat{f}(s_{(p)}) \leq \hat{f}(s_{(k)}) \right\}} \, \text{ and } \, \sigma_S(s_{(k)}) = k.
\]

For any integer $p \leq k - \frac{\epsilon}{2} N\Delta_N$ we have $s_{(p)} \leq s_{(k)} - \frac{\epsilon}{2} \Delta_N$, thus by Lemma~\ref{lemma:almost-correct-ordering-function}
\begin{equation*}
    \forall p \leq \left \lfloor  k - \frac{\epsilon}{2} N\Delta_N  \right \rfloor \, : \, \hat{f}(s_{p}) < \hat{f}(s_{(k)}) \hspace{0.5cm} \textit{w.e.p.}
\end{equation*}
Therefore 
\begin{equation*}
    \hat{\sigma}_S(s_{(k)}) \geq \left \lfloor  k - \frac{\epsilon}{2} N\Delta_N  \right \rfloor \geq k - \epsilon N \Delta_N \hspace{0.5cm} \textit{w.e.p.}
\end{equation*}
Similarly for any $ p \geq k + \frac{\epsilon}{2} N\Delta_N$ we have $s_{(p)} \geq s_{(k)} + \frac{\epsilon}{2} \Delta_N$, thus by Lemma~\ref{lemma:almost-correct-ordering-function}
\begin{equation*}
    \forall p \geq \left \lceil  k + \frac{\epsilon}{2} N\Delta_N  \right \rceil   \, : \, \hat{f}(s_{p}) > \hat{f}(s_{(k)}) \hspace{0.5cm} \textit{w.e.p.}
\end{equation*}
Therefore 
\begin{equation*}
    \hat{\sigma}_S(s_{(k)}) \leq \left \lceil  k + \frac{\epsilon}{2} N\Delta_N  \right \rceil \leq k + \epsilon N \Delta_N \hspace{0.5cm} \textit{w.e.p.}
\end{equation*}
Finally,
\begin{equation*}
    \left(\forall v \in S \right) \, : \, \left|\hat{\sigma}_S(v) - \sigma_S(v)\right| \leq \epsilon N\Delta_N = \epsilon \sqrt{N\log^{\gamma}(N)} \hspace{0.5cm} \textit{w.e.p},
\end{equation*}
as required.
\end{proof}

\subsection{Proof of Theorem \ref{ThmSomePostprocMain} and analysis of Algorithm \ref{AlgPostProc}}

We now prove the main bounds on the objects appearing in Algorithm \ref{AlgPostProc}, completing the proof of Theorem \ref{ThmSomePostprocMain}.  Throughout this section, we use notation from  Algorithm \ref{AlgPostProc} and make the same assumptions as in the statement of Theorem \ref{ThmSomePostprocMain}. By Theorem~\ref{thm:error-bound-alg2},
\begin{equation} \label{IneqSummaryMainAlg}
\| \hat{\sigma}_{j} - \sigma_{S_{j}} \|_{\infty} \leq \frac{\epsilon}{12} \, \sqrt{ N \log^{\gamma}(N)}
\end{equation}

holds for all $j \in \{1,2,3\}$ \textit{w.e.p.}

We next discuss Step 3 of the algorithm. Recall the definition of \textit{aligned} and \textit{closely aligned} from Inequality \eqref{EqAlignDef}. Note that $\{ \hat{\sigma}_{j} \}_{j \in \{1,2,3\}}$ were computed in separate calls to Algorithm \ref{AlgPostProcSplit} and so they may not be aligned - that is, some $\hat{\sigma}_j$ may be aligned with $\sigma_{S_j}$ while others may be aligned more with its reverse $\mathrm{rev}(\sigma_{S_j})$.   When $N$ is large enough all the $\sigma_{S_j}$'s are closely aligned with $id_N$. Moreover, when Inequality \eqref{IneqSummaryMainAlg} holds and $N$ sufficiently large, $\hat{\sigma}_{j}$ will either be closely aligned with $\sigma_{S_j}$ or with $\mathrm{rev}(\sigma_{S_j})$. That is, $\hat{\sigma}_{j}$ will be be closely aligned with $id_N$ or with its reverse for all $j \in \{1,2,3\}$.  For the same reason, for $N$ sufficiently large $\hat{\sigma}_{\mathrm{ref}}$ will either be closely aligned with $id_{N}$ or with $\mathrm{rev}(id_{N})$. We assume without loss of generality that $\hat{\sigma}_{\mathrm{ref}}$ is closely aligned with $id_{N}$. Finally, note that \textit{w.e.p.}, $|S_{j} \cap S'_{2}| \geq 2$ for all $j \in \{1,2,3\}$. When all of these events occur, for each $j \in \{1,2,3\}$, one of the following two cases must occur:

\begin{enumerate}
    \item \textbf{$\hat{\sigma}_{j}$ is closely aligned with $id_{N}$.} Since $|S_{j} \cap S'_{2}| \geq 2$ and $\hat{\sigma}_{j}$, $\hat{\sigma}_{\mathrm{ref}}$ are both closely aligned to $id_{N}$, they must also be closely aligned to each other. Thus, in this case, $\hat{\sigma}_{j}$ \textit{is not} reversed in step 3 of the algorithm and \textit{remains} closely aligned to $id_{N}$ in step 4.
    \item \textbf{$\hat{\sigma}_{j}$ is closely aligned with $\mathrm{rev}(id_{N})$.} Since $|S_{j} \cap S'_{2}| \geq 2$ and $\mathrm{rev}(\hat{\sigma}_{j})$, $\hat{\sigma}_{\mathrm{ref}}$ are both closely aligned to $id_{N}$, they must also be aligned aligned to each other. Thus, in this case, $\hat{\sigma}_{j}$ \textit{is}  reversed in step 3 of the algorithm and \textit{becomes} closely aligned to $id_{N}$ in step 4.
\end{enumerate}

Thus, in both cases, by step 4 the ordering $\hat{\sigma}_{j}$ is closely aligned to $id_{N}$ for all $j \in \{1,2,3\}$ \textit{w.e.p.}

Denote by $s_1^{(k)}, s_2^{(k)}$ and $s_3^{(k)}$ the order statistics of the sets $S_1, S_2$ and $S_3$ respectively. Then by Equation~\eqref{eq:merged-permutation} we have
\begin{align*}
    \norm{\hat{\sigma} - id_N}_{\infty}
    & =  \norm{\hat{\sigma}^{-1} - id_N}_{\infty} \\
    & = N \max_{j \in \{1,2,3\}} \max_{k \in \{1, \ldots, n+1\}} \left|\hat{\sigma}_{j}^{-1}(k)-v_{3k+j-3}\right| \\ 
    & \leq N \max_{j \in \{1,2,3\}} \max_{k \in \{1, \ldots, n+1\}} \left|\hat{\sigma}_{j}^{-1}(k)-v_{3k}\right| + 2 \\
    & =  \max_{j \in \{1,2,3\}} \max_{k \in \{1, \ldots, n+1\}} \left|3\hat{\sigma}_j(s_j^{(k)}) - N s_j^{(k)}\right| + 2 \\
    & \leq 3 \max_{j \in \{1,2,3\}} \norm{\hat{\sigma}_{j} - \sigma_{S_j}}_{\infty} 
    + N \max_{j \in \{1,2,3\}} \max_{k \in \{1, \ldots, n+1\}} \left|s_j^{(k)} - v_{3k}\right| + 2.
\end{align*}

By Lemma~\ref{lemma:order-stat-bound}, 
\begin{equation}\label{eq:order-stat-ineq}
    \left|s_j^{(k)} - v_{3k}\right| \leq \frac{\epsilon}{4}\Delta_N \hspace{0.5cm} \textit{w.e.p},
\end{equation}
for all $j \in \{1,2,3\}$ and $k \in \{1,\ldots, n+1\}$ such that $3k+j \leq N$. When $N$ is large enough, and when inequalities~\eqref{IneqSummaryMainAlg} and \eqref{eq:order-stat-ineq} hold, we get
\begin{equation*}
    \norm{\hat{\sigma} - id_N}_{\infty} \leq \epsilon\sqrt{N\log^{\gamma}(N)} \hspace{0.5cm} \textit{w.e.p.}
\end{equation*}
which completes the proof.

\begin{remark}[Learning $\alpha$ and $\beta$] \label{RemFindingAlphaBeta}

Algorithm \ref{AlgPostProcSplit} takes parameters $\alpha, \beta$ as given. However, our results only apply if these parameters satisfy Assumption \ref{AssumptionPsiOK}. In Theorem \ref{ThmSimpleVersion}, we state values of $(\alpha,\beta)$ for which the graphons satisfying Assumption \ref{Assumption:weakfamily} will also satisfy  Assumption \ref{AssumptionPsiOK}, and so there are no difficulties in that setting. 

However, for other graphons, it may not be obvious how to choose $\alpha, \beta$. We note here that, when graphons satisfy a slightly strengthened version of the assumption for some pair $\alpha,\beta$, it is not difficult to learn this pair quickly. We now give details. 

Note that the functions $\Psi_{R} = \Psi_{R}^{(\alpha)}, \Psi_{L} = \Psi_{L}^{(\alpha)}$ depend on the parameter $\alpha$ (though this dependence is suppressed in our notation). Recall that $\Psi_{R}, \Psi_{L}$ are continuous functions. If $w$ satisfies  Inequality    \eqref{IneqMDI} for some $\alpha < \beta$, then there exists some $\delta > 0$ and some other pair $\alpha', \beta'$ so that it satisfies the slightly strengthened bound
   \begin{equation}\label{IneqMDIStronger}
        \inf_{x \in [1-\alpha'-2\delta,1]} \Psi_{R}^{(\alpha')}(x) > \Psi_{R}^{(\alpha')}(1-\beta') + 2\delta, \qquad \inf_{x \in [0,\alpha' + 2\delta]} \Psi_{L}^{(\alpha')}(x) > \Psi_{L}^{(\alpha')}(\beta') + 2\delta,
    \end{equation}
By Lemma \ref{lemma:psiR-PsiR-bound} the estimates $\psi_{R}$, $\psi_{L}$ calculated in Equations \eqref{eq:psiR-estimate}, \eqref{eq:psiL-estimate} have error $\Delta_{N}$ going to 0 with $N$. Thus, for all $N$ sufficiently large, with extreme probability any pair $(\alpha',\beta')$ satisfying \eqref{IneqMDIStronger} will also satisfy
   \begin{equation}\label{IneqMDIStrongerEmpirical}
        \inf_{x \in [1-\alpha'-\delta,1]} \psi_{R}^{(\alpha')}(x) > \psi_{R}^{(\alpha')}(1-\beta') + \delta, \qquad \inf_{x \in [0,\alpha' + \delta]} \psi_{L}^{(\alpha')}(x) > \psi_{L}^{(\alpha')}(\beta') + \delta. 
    \end{equation}

Applying Lemma \ref{lemma:psiR-PsiR-bound} again, any pair satisfying the Inequality \eqref{IneqMDIStrongerEmpirical} will also satisfy Inequality \eqref{IneqMDI}.

Thus, if a graphon satisfies \eqref{IneqMDIStronger} for some $\delta > 0$ and some pair $(\alpha,\beta)$, we can find the pair $(\alpha',\beta')$ by using Inequality \eqref{IneqMDIStrongerEmpirical} as a test. Since the bound in Lemma \ref{lemma:psiR-PsiR-bound} is a bound on the \textit{maximum} error, checking \textit{many} pairs $(\alpha',\beta')$ from a single observed graph does not present any obstacles.

\end{remark}
\section{Application} \label{SecApplNice}

We show that Theorem \ref{ThmSimpleVersion} follows from Theorems \ref{ThmMainThm} and \ref{ThmSomePostprocMain}. To do so, it is enough to check that any graphon that is ``nice" in the sense of Definition \ref{Assumption:weakfamily} will also satisfy Assumptions \ref{assumption:lipschitz} through \ref{AssumptionPsiOK}.

We now check Assumptions \ref{assumption:lipschitz} through \ref{AssumptionPsiOK}  in order.

\textbf{Assumption \ref{assumption:lipschitz}:} Define $R_{\max} = \sup_{r \in (0,1]} R'(r) < \infty$. We calculate
\begin{align*}
    \abs{w(x,y)-w(x',y')} &= \abs{R(|x-y|) - R(|x'-y'|)} \\
    & \leq R_{\max} \abs{ \, |x-y| - |x'-y'| \, } \\
    &\leq R_{\max} \abs{(x-y) - (x'-y')} \\
    &\leq R_{\max} ( |x-y| + |x'-y'| ).
\end{align*}

\textbf{Assumption \ref{assumption:PartialDerivativeGraphon}:} We have $\frac{\partial w(x,y)}{\partial x} = R'(y)$ for $y \neq x$, so this follows immediately from Definition \ref{Assumption:weakfamily}.

\textbf{Assumption \ref{assumption:NonZerod'}:} We calculate 
\begin{align*}
    d'(x) &= \frac{d}{dx} \int_{0}^{1}R(|x-y|) dy \\
    &= \frac{d}{dx} \left( \int_{0}^{x} R(y)dy + \int_{0}^{1-x} R(y)dy  \right) \\ 
    &= R(x) - R(1-x).
\end{align*}
Thus $d'(x)$ exists for all $x \in [0,1]$. Since $R$ is monotone, $d'(x) > 0$ for $x \in [0,0.5)$ and $d'(x) < 0$ for $x \in (0.5,0]$, with $d'(x) = 0$ only for $x = 0.5$.

\textbf{Assumption \ref{assumption:InfimumInequality}:} This is given in \cite{MOPinelis}, with the function $h$ in that reference replacing the function $R$ in this paper.

\textbf{Assumption \ref{AssumptionPsiOK}:} To prove Inequality \eqref{IneqMDI}, we have
\begin{align*}
    \inf_{x \in [0,0.05]} \Psi_{L}(x) &= \inf_{x \in [0,0.05]} \int_{0}^{0.05} w(x,y) dy \\
    &= \inf_{x \in [0,0.05]} \int_{0}^{0.05} R(|x-y|) dy.
\end{align*}
Since $R$ is monotone, this infimum is achieved at $x=0$ and $x=0.05$. Continuing the calculation,  
\begin{align*}
    \inf_{x \in [0,0.05]} \Psi_{L}(x) &= \int_{0}^{0.05} R(y) dy \\
    &> \int_{0.26}^{0.31} R(y) dy \\ 
    &= \Psi_{L}(0.31),
\end{align*}

where the strict inequality in the second line comes from the assumption that $R'$ is continuous and strictly positive. This proves one of the four inequalities in the first part Assumption \ref{AssumptionPsiOK}; the remainder are essentially the same.

To prove Inequality \eqref{ineq:weaker-assumption},  recall that $\frac{\partial w(x,y)}{\partial x} = R'(y) > 0$. Thus, this follows immediately from Gronwall's inequality.

\section{Discussion} \label{SecOptDisc}

In this paper, we have written down a simple seriation algorithm and shown that, for some fairly large class of model graphons, the result of the algorithm is an estimated permutation $\hat{\sigma}$ with sup-norm error of roughly $O(\sqrt{N})$. It is natural to ask whether either the rates of convergence or our assumptions about the generating graphon  are optimal.

The second question seems to have an easy answer: we are quite sure that our conditions are not optimal, even for the type of analysis we use. In particular, our main assumptions on $w$ imply that all vertices have a fairly non-negligible chance of being connected (more precisely, our main assumptions imply  that $ w(x,y) > 0$ for all $x,y \in [0,1]$). We have already seen in Remark \ref{RemToyGraph} that this property is not satisfied for many natural graphons, and also that it can be substantially weakened. 

We hypothesize that this global connection property can be replaced by a much weaker assumption, saying roughly that some \textit{power} of the adjacency matrix should be globally connected. The main difficulties in carrying this out are different for our two main results:
\begin{itemize}
    \item The main difficulty in extending Theorem \ref{ThmMainThm} is carrying through the arguments in Section \ref{SecFiedlerProp}, which show that the Fiedler eigenspace has dimension 1 and that the Fiedler function is monotone. 
    
    As discussed in Remark \ref{RemToyGraph}, this can sometimes be checked directly for particularly simple graphons. We also expect that certain generic techniques, such as perturbation analysis, can be used to extend this to small nonparametric families of graphons. However, we do not know how to get rid of this assumption in any great generality.
    \item The main difficulty in extending Theorem \ref{ThmSomePostprocMain} is not merely technical: Algorithm \ref{AlgPostProc} will not generally give a sensible answer without the global connectedness property. This occurs because the functions $\psi_{R}(i),\psi_{L}(i)$ may both be 0 for many $i \in [N]$. In \cite{janssen2019reconstruction}, we introduce a ``local" quantity $R(i,j)$ that is analogous to the ``global" quantity $R$ used in the present paper. We suspect that a similar approach would work in the present context.
\end{itemize}

We don't know the optimal rate of convergence for this problem in very many metrics, but we give a few related results largely drawn from our companion paper \cite{janssen2019reconstruction}. Note that this paper studies a slightly different sampling scheme, but very similar arguments apply:

\begin{enumerate}
    \item Under slightly different model conditions, the iterated post-processing algorithm of \cite{janssen2019reconstruction} can obtain error of $O(N^{\epsilon})$ for any fixed $\epsilon > 0$, so the rate $N^{0.5}$ is certainly not optimal for all classes of graphon models.
    \item On the other hand, we might ask to estimate the latent positions $v_{1},\ldots,v_{N}$ of vertices in addition to the latent permutation $\sigma$. The argument in Section 5 of \cite{janssen2019reconstruction} says that any such estimate must suffer an error of at least $\Omega(N^{-0.5})$ for any sufficiently large class of graphon models. Since the natural estimate
    \[ 
    \hat{v}_{i} = \frac{\hat{\sigma}(i)}{N}
    \]
    is so similar to $\hat{\sigma}$ itself, this suggests that one can't do much better in general.
\end{enumerate}

Several other papers have recently studied the optimal rate of convergence for ``the" noisy seriation problem, including \cite{mao2018breaking} and \cite{flammarion2019optimal}. Their models and versions of the noisy seriation problem are different enough from ours that it is not possible to directly compare the results. We leave the question of comparing these notions of noisy seriation to future work.

\subsection*{Acknowledgements}

We thank Iosif Pinelis for his detailed response \cite{MOPinelis} on MathOverflow. The response was quite a bit of work, given without any request for reward, and let us substantially simplify the assumptions appearing in Section \ref{SecNotation}.
\medskip
\bibliographystyle{imsart-number}
\bibliography{references}
\nocite{*}
\newpage
\appendix

\section{Appendix}
\begin{lemma}\label{lemma:order-stat-bound}
Let $V = \{v_1, v_2, \ldots, v_N\}$ with $v_i = i/N$ for all $i$. Let $S$ be a subset of $V$ chosen uniformly at random and with deterministic size $m = |S|$ such that $\left \lfloor \tfrac{N}{3} \right \rfloor \leq m \leq \left \lfloor \tfrac{N}{3} \right \rfloor + 1$. Define $s_{(1)} < s_{(2)} < \ldots < s_{(m)}$ to be the order statics of the elements of $S$. Then for each $1 \leq k \leq m$ and for any $c > 0$, there exists a constant $\delta > 0$ and an integer $N_0$ such that
\begin{equation*}
    \Prob{\left| s_{(k)} - \frac{k}{m} \right| > c \Delta_N} \leq \varepsilon_{\gamma}^{\delta}(N)
\end{equation*}
for all $N \geq N_0$. In other words for each $1 \leq k \leq m$ and $c > 0$ the event $\left| s_{(k)} - \frac{k}{m} \right| \leq c\Delta_N$ holds \textit{w.e.p}.
\end{lemma}
\begin{proof}[Proof of Lemma~\ref{lemma:order-stat-bound}]
Let $1 \leq k \leq N$ and $ k \leq j \leq N - m + k$. Then $s_{(k)} = v_j$ if and only if there is exactly $k-1$ elements from $S$ less than $v_j$ and exactly $m-k$ elements from $S$ bigger than $v_j$. Therefore
\begin{equation*}
    \Prob{s_{(k)} = v_j} = \frac{\binom{j-1}{k-1}\binom{N-j}{m-k}}{\binom{N}{m}}.
\end{equation*}
Note that the distribution of $s_{(k)}$ looks very similar to a negative hypergeometric distribution. In fact one can easily check that $N s_{(k)} - k \sim \mathcal{NHG}(N; N-m; k)$. 

The paper \cite{serfling1974probability} provides concentration inequalities for sampling without replacement. According to \cite{serfling1974probability, greene2017exponential}, for any $\lambda > 0$ we have
\[
\Prob{|s_{(k)} - \mathbb{E}[s_{(k)}]| > \lambda} \leq 2 \exp \left( \frac{-2 N^2 \lambda^2}{m(1-f_m^*)}\right)
\]
with $f_m^* = (m-1)/N$. We can check that $1/4 < f_m^* \leq 1/3$ for all $N$ sufficiently large. Thus,
\[
\Prob{|s_{(k)} - \mathbb{E}[s_{(k)}]| > \lambda} \leq 2 \exp \left( \frac{-8 N^2 \lambda^2}{3m}\right).
\]
Let $c > 0$, we have
\begin{align*}
    |s_{(k)} - \frac{k}{m}| > c \Delta_N 
    & \Rightarrow |s_{(k)} - \mathbb{E}[s_{(k)}]| > c\Delta_N - |\frac{k}{m}-\mathbb{E}[s_{(k)}]| \\
    & \Rightarrow |s_{(k)} - \mathbb{E}[s_{(k)}]| > c\Delta_N - \frac{(N-m)k}{Nm(m+1)} \\
    & \Rightarrow |s_{(k)} - \mathbb{E}[s_{(k)}]| > c\Delta_N - \frac{2}{N}.
\end{align*}
For $N$ large enough so that $4 < c N \Delta_N$,
\begin{equation*}
    |s_{(k)} - \frac{k}{m}| > c \Delta_N \Rightarrow |s_{(k)} - \mathbb{E}[s_{(k)}]| > \frac{c}{2}\Delta_N.
\end{equation*}
Therefore 
\begin{align*}
    \Prob{ |s_{(k)} - \frac{k}{m}| > c \Delta_N } 
    & \leq \Prob{|s_{(k)} - \mathbb{E}[s_{(k)}]| > \frac{c}{2}\Delta_N} \\ 
    & \leq 2 \exp \left( \frac{-2 c^2 N^2 \Delta_N^2}{3m}\right) \\
    & \leq 2 \exp \left(- c^2 N \Delta_N^2\right) \\ 
   % & = 2  \exp \left(- c^2 \log^{\gamma}(N)\right) \\
    & \leq \varepsilon_{\gamma}^{\delta}(N)
\end{align*}
for some $0 < \delta < c^2$.
\end{proof}

\end{document}